\documentclass[a4paper,11pt]{article}
\usepackage{amsmath,amssymb,amsthm}
\usepackage{mathtools}
\usepackage{microtype}
\usepackage{hyperref}
\usepackage{ytableau}

\usepackage{tikz}

\newdimen\tabcellsize\newdimen\halftabcellsize
\newcount\tabrownum\newcount\tabcolnum

\newcommand\mytab[2][5mm]{
 \setlength{\tabcellsize}{#1}
 \halftabcellsize = 0.5\tabcellsize
 \begin{tikzpicture}[baseline=(current bounding box.center)]
 \node at (0,0) {\makebox[0pt][0pt]{\phantom{$#1$}}};
 \tabrownum=0
 \foreach\rowlist in #2{
 	\draw (0,\the\tabrownum*\the\tabcellsize) -- ++(0,\the\tabcellsize);
 	\draw (0,\the\tabrownum*\the\tabcellsize+\the\tabcellsize) -- ++(1/4*\the\halftabcellsize,0);
 	\tabcolnum=0
 	\foreach\colitem in \rowlist{
 		\draw (\the\tabcolnum*\the\tabcellsize,\the\tabrownum*\the\tabcellsize) -- ++(\the\tabcellsize,0);
 		\draw (\the\tabcolnum*\the\tabcellsize,\the\tabrownum*\the\tabcellsize+\the\tabcellsize) -- ++(\the\tabcellsize,0);
 		\draw (\the\tabcolnum*\the\tabcellsize+\the\tabcellsize,\the\tabrownum*\the\tabcellsize) -- ++(0,\the\tabcellsize);
 		\node at (\the\tabcolnum*\the\tabcellsize+\the\halftabcellsize,\the\tabrownum*\the\tabcellsize+\the\halftabcellsize)
                {\colitem};
         	\global\advance\tabcolnum by 1
 	}
        \global\advance\tabrownum by -1
 }
 \end{tikzpicture}
}

\DeclareMathOperator*{\bigboxtimes}{\raisebox{-3pt}{\scalebox{1.7}{$\boxtimes$}}}

\newcommand{\tss}[1]{\textsuperscript{#1}}
\newcommand{\ul}[1]{\underline{#1}}
\newcommand{\cDot}{\boldsymbol{\cdot}}

\newtheorem{theorem}{Theorem}
\newtheorem{proposition}[theorem]{Proposition}
\newtheorem{corollary}[theorem]{Corollary}
\newtheorem{lemma}[theorem]{Lemma}

\title{Specht module branching rules for wreath products of symmetric groups}
\author{\begin{tabular}{c}Reuben Green\thanks{Supported by EPSRC grant [EP/M508068/1]} \\ \texttt{\scriptsize reubengreen183@gmail.com}\end{tabular}\\[1em]
\small School of Mathematics, Statistics and Actuarial Science,\\
\small University of Kent, CT2 7NF, UK}

\begin{document}
\maketitle

\begin{abstract}
{\noindent}
We review a class of modules for the wreath product $S_m \wr S_n$ of two symmetric groups which are analogous to the Specht modules
of the symmetric group, and prove a pair of branching rules for this family of modules. These branching rules describe the behaviour
of these wreath product Specht modules under restriction to the wreath products $S_{m-1} \wr S_n$ and $S_m \wr S_{n-1}$. In particular, we
see that these restrictions of wreath product Specht modules have Specht module filtrations, and we obtain combinatorial interpretations
of the multiplicities in these filtrations.
\end{abstract}

\section{Introduction}
Let $k$ be a field. Recall that the \textit{Specht modules} for the symmetric group $S_n$ (over $k$) are a family of $kS_n$-modules
(we shall use right modules in this article) which are indexed by the partitions of $n$. We shall write the Specht module for $S_n$ which
is indexed by the partition $\lambda$ as $S^\lambda$. These Specht modules have a close relationship with the simple modules of
$kS_n$, and indeed if $kS_n$ is semisimple then the Specht modules are exactly the simple modules. Because of this and other properties,
the Specht modules for $kS_n$ have been the subject of intense study for decades, and a large and varied literature has built up around
them.

In this article we consider a class of modules for the \textit{wreath product} $S_m \wr S_n$ of two symmetric groups which are
analogous to the Specht modules for the symmetric group. These modules may be obtained from Specht modules for $S_m$ and $S_n$ via
a well-known method of constructing modules for wreath products, see for example \cite{CHTAN} or \cite[chapter 4]{JAMKER}.
We may alternatively obtain them as the cell modules of a certain cellular structure (in the sense of Graham and Lehrer) on the group
algebra $k(S_m \wr S_n)$. This cellularity was originally proved in \cite{GEGO}, while an alternative proof via the method of
\textit{iterated inflation} was given in \cite{GCSWPA}. The characterisation of these modules as cell modules
allows us to see at once that they bear exactly the same relation to the simple modules for the wreath product as the symmetric group
Specht modules bear to the simple modules for that group, and this justifies the name ``Specht modules''. Although the construction by
which these modules may be obtained is well-known, the author is not aware that these modules have previously been studied in the
literature as wreath produce analogues of the symmetric group Specht modules.

A key fact in the theory of Specht modules is the result of James which we shall call the ``Specht branching rule'',
which gives a Specht filtration for the restriction of a Specht module from $kS_n$ to $kS_{n-1}$ with an
elegant combinatorial description of the set of Specht modules occurring in this filtration. Moreover, these multiplicities are
independent of the field $k$. The main results presented in this paper are two Specht branching rules for the wreath product of two
symmetric groups: the first describes the restriction of a wreath product Specht module from $k(S_m \wr S_n)$ to $k(S_{m-1} \wr S_n)$,
while the second describes the restriction to $k(S_m \wr S_{n-1})$. In both cases, we obtain a Specht module filtration with
multiplicities that do not depend on the field and which moreover have nice combinatorial descriptions.

Note that we are using the name ``branching rule'' in what is perhaps a slightly non-standard way. Indeed, in general group
representation theory, if we have some nested family of finite groups $G_1 \leq G_2 \leq \cdots \leq G_n \leq G_{n+1} \cdots$, then a
\textit{branching rule} is a result describing the simple composition factors of the restriction of a simple module from $G_{n+1}$ to $G_n$.
However, our results are of a very similar nature, and indeed
our Specht modules are in fact the simple modules when the group algebra is semisimple (in both the symmetric group and
the wreath product case) and so in the semisimple case our Specht branching rules are in fact branching rules in the more usual sense.

\section{Background}
We shall work over a field $k$ in this article. We shall often need to deal with tensor products of $k$-vector spaces, and we shall
abbreviate $\otimes_k$ to $\otimes$.

If $G$ is a group and $k$ is a field, then we shall write $kG$ for the group algebra of $G$ over $k$. By a \textit{$kG$-module}, we shall
mean a right $kG$-module of finite $k$-dimension.
If $G$ is a group with a subgroup $H$, then for a field $k$ we shall write the operations of induction and restriction of modules
between the group algebras $kG$ and $kH$ as $\uparrow^{G}_{H}$ and $\downarrow^{G}_{H}$, with the field being implicit.

Mackey's theorem is a fundamental result in finite group theory which describes the interaction of the operations of induction and
restriction. If $G$ is a finite group and $H$ is a subgroup of $G$, then for $g \in G$ we define $H^g$ to be the subgroup
$\{g^{-1}hg \mid h \in H\}$ of $G$, and we call this the \textit{conjugate subgroup of $H$ by $g$} (note that $H^g$ is isomorphic to $H$).
Further, if $X$ is a $kH$-module, then we define $X^g$ to be the $kH^g$-module with underlying vector space $X$ and action given by
$x(g^{-1}hg) = xh$ for $x \in X$ and $h \in H$. We call this the \textit{conjugate module of $X$ by $g$}.

\begin{theorem}\label{mackey:thm}(Mackey's Theorem)\cite[Theorem 3.3.4]{BRCI}
Let $G$ be a finite group with subgroups $H$ and $K$, let $\mathcal{U}$ be a complete non-redundant system of $(H,K)$-double coset
representatives in $G$, and let $X$ be a right $kH$-module. Then we have a decomposition of right $kK$-modules
\[
X\!\uparrow^{G}_{H}\downarrow^{G}_{K} \;\;\cong\;
\bigoplus_{u \in \mathcal{U}} X^u\downarrow^{H^u}_{H^u\cap K}\uparrow^{K}_{H^u\cap K}.
\]
\end{theorem}

\subsection{Filtrations}
Let $kG$ be a group algebra over a field, let $M$ be a $kG$-module and $X_1,\ldots,X_t$ also be $kG$-modules. We say that $M$ has a
\textit{filtration by the modules $X_1,\ldots,X_t$} if there is a chain of submodules
\[ M = M_n \supseteq M_{n-1} \supseteq M_{n-2} \cdots \supseteq M_1 \supseteq M_0 = 0 \]
such that each quotient $\frac{M_l}{M_{l-1}}$ is isomorphic to some $X_i$. If $\frac{M_1}{M_0} = M_1$ is isomorphic to $X_l$, then we
say that $X_l$ occurs at the \textit{bottom} of the filtration.
Now suppose that for each $i=1,\ldots,t$, $\alpha_i$ is a non-negative integer. We shall say that $M$ has a filtration by the modules
$X_1,\ldots,X_t$ where $X_i$ \textit{has multiplicity $\alpha_i$} if there is a chain of submodules as above
and a function $f:\{1,\ldots,n\} \longrightarrow \{1,\ldots,t\}$ such that for each $l$, $X_{f(l)}$ is isomorphic to the quotient of
$M_l$ by $M_{l-1}$, and $|f^{-1}(i)| = \alpha_i$ for each $i$.
Note that in the above definitions, we have not assumed that the modules $X_i$ are pairwise non-isomorphic. If there are isomorphisms
between the modules $X_i$, then the multiplicities in a filtration are not uniquely determined by the chain of submodules, and so the
same chain of submodules can be considered to give rise to filtrations with different multiplicities. Even if the modules $X_i$ are
pairwise non-isomorphic, so that the multiplicities \textit{are} uniquely determined by the filtration, the multiplicities are
\textit{not} in general uniquely determined by the module, as the same module can have two chains of submodules where the $X_i$ occur
with different multiplicities.

\subsection{Combinatorics}
We now review a few combinatorial concepts. We assume that the reader is already familiar with these notions, and so our treatment will
be brief.

Recall that a \textit{composition} of $n$ is a tuple of non-negative integers adding up to $n$. We call $n$ the \textit{size} of $\alpha$
and write $n = |\alpha|$. We call the elements of a composition its \textit{parts}. We shall adopt the common shorthand of using exponent
notation for repeated parts in a composition, so that for example we might write $(3,2^2,1^3)$ for $(3,2,2,1,1,1)$.
A \textit{partition} of $n$ is a composition of $n$ whose parts are all positive and appear in non-increasing order. We shall write
$\lambda \vdash n$ to mean that $\lambda$ is a partition of $n$.
A simple total order on the partitions of an integer $n$ is the \textit{lexicographic order}, in which partitions are sorted by the
size of their first part, then by the size of their second part, and so on. Thus in this order, $(n)$ is the largest and $(1^n)$ the
least partition.

The \textit{Young diagram} of a composition $\alpha$ is an arrangement of rows of boxes with a number of boxes on the $i$\tss{th} row
(counting downward) equal to the $i$\tss{th} part of $\alpha$.
If $\alpha,\gamma$ are compositions of $n$, then a \textit{tableau of shape $\alpha$ and type $\gamma$} is a Young diagram of shape
$\alpha$ where each box contains a positive integer $i$ such that for each $i \in \{1,\ldots,t\}$ where $t$ is the length of $\gamma$,
$i$ occurs exactly $\gamma_i$ times.
Note that since we allow zero parts in compositions, a Young diagram or tableau can have empty rows.

Now if $\lambda,\alpha$ are compositions such that $|\alpha|\leq|\lambda|$ and the Young diagram of $\alpha$ lies wholly inside the Young
diagram of $\lambda$ (i.e. the length of $\alpha$ is at most the length of $\lambda$ and $\alpha_i \leq \lambda_i$ for all $i$ from 1 to
the length of $\alpha$), then for $\gamma$ a composition of $|\lambda|-|\alpha|$, we define a \textit{skew tableau of shape $\lambda
\setminus \alpha$ and type $\gamma$} to be a diagram obtained by removing the boxes of the Young diagram of $\alpha$ from $\lambda$ and
then filling the remaining boxes with positive integers such that each $i$ occurs $\gamma_i$ times, as for non-skew tableaux. Note that
the boxes of a skew tableau may be non-contiguous, as in the example below.

If a tableau (skew or non-skew) has its entries strictly increasing down each column and weakly increasing from left to right across each
row, we say that it is \textit{semistandard} (note that there may be gaps in the columns of a skew tableau). Thus for example
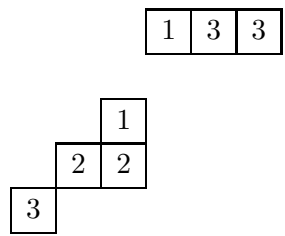
\[\ytableausetup{notabloids}
\begin{ytableau}
\none & \none & \none & 1 & 3 & 3\\
\none\\
\none & \none & 1\\
\none & 2 & 2\\
3\\
\end{ytableau}\]
is a semistandard skew tableau of shape $(6,4,3,3,1)\setminus(3,4,2,1)$ and type $(2,2,3)$.

A \textit{multicomposition} of $n$ is a tuple of compositions whose sizes add up to $n$. A multicomposition whose components are all
partitions is called a \textit{multipartition}. We typically use underlined symbols to denote multicompositions and index their components
with superscripts, so that for example a multicomposition of length $t$ might be written $\ul{\alpha}$, with
$\ul{\alpha} = (\alpha^1,\ldots,\alpha^t)$, and $\alpha^i_j$ being the $j$\tss{th} part of the composition $\alpha^i$. The \textit{size}
of a multicomposition is the sum of the sizes of its parts, and if
$\ul{\alpha} = (\alpha^1,\ldots,\alpha^t)$ is a multicomposition of $n$, then we let $|\ul{\alpha}|$ be the composition
$(|\alpha^1|,\ldots,|\alpha^t|)$ of $n$.

\subsection{Symmetric groups}
If $\alpha$ is a composition of $n$, then we shall write $S_\alpha$ for the \textit{Young subgroup} of $S_n$ associated to $\alpha$.
We shall be making frequent use of the operations of induction and restriction between group algebras of symmetric groups and of their
Young subgroups, for example
$\left.X\right\uparrow^{kS_n}_{kS_\alpha}$ and $\left.Y\right\downarrow^{kS_n}_{kS_\alpha}$.
To de-clutter such expressions, we shall abbreviate the notation by replacing the full symbols for the group algebras with the subscripts
used to identify the various subgroups of $S_n$ involved, so for example the above would be abbreviated to
$\left.X\right\uparrow^{n}_{\alpha}$ and $\left.Y\right\downarrow^{n}_{\alpha}$.

Recall that we have for each $n > 0$ a natural embedding of the symmetric group $S_{n-1}$ into $S_n$ by letting
$\sigma \in S_{n-1}$ act on ${1,\ldots,n}$ by fixing $n$ and permuting the other elements as it does in $S_{n-1}$. Thus we can regard
$S_{n-1}$ as a subgroup of $S_n$ and hence we may induce a module $X$ from $kS_{n-1}$ to $kS_n$, or restrict a module $Y$ from $kS_n$ to
$kS_{n-1}$. We shall write these operations as
$\left.X\right\uparrow^{n}_{n-1}$  and $\left.Y\right\downarrow^{n}_{n-1}$.

If a $kS_n$-module $M$ has a filtration by the Specht modules $S^\lambda$ for $\lambda \vdash n$, then we say that $M$ has a
\textit{Specht module filtration}, or just a \textit{Specht filtration}.

\begin{theorem}(Specht branching rule)\label{jam_branch_rule_sn:thm}
(\cite{JAMSN}, Theorem 9.3) Let $\lambda \vdash n$ where $n > 0$, and let $k$ be a field. Then the $kS_{(n-1)}$-module
$\left.S^\lambda\right\downarrow^n_{n-1}$ has a Specht filtration where for $\nu \vdash (n-1)$, $S^\nu$ has multiplicity one if the Young
diagram of $\nu$ can be obtained from the Young diagram of $\lambda$ by removing a single box, and $S^\nu$ has multiplicity zero
otherwise.
\end{theorem}

We shall make use below of \textit{Littlewood-Richardson coefficients}. These numbers appear in many different places in combinatorics and
representation theory and have an extensive literature, but we shall only need a few basic facts. The reader is referred to the
literature, for example \cite[chapter 7]{STANEC}, for more details. Indeed, if $\lambda$ is a partition
of $n$ and $\alpha, \beta$ are partitions whose sizes add up to $n$, then we have a non-negative integer $c^\lambda_{\alpha,\beta}$ called
a \textit{Littlewood-Richardson coefficient}. Moreover, if $(\alpha_1,\ldots,\alpha_t)$ for $t \geq 1$ is a tuple of partitions whose
sizes add up to $n$, then we may define a more general Littlewood-Richardson coefficient $c(\lambda;\ul{\alpha})$. Indeed, for the case
$t = 1$, we let $c(\lambda;\ul{\alpha}) = c\bigl(\lambda;(\alpha^1)\bigr)$ be 1 if $\ul{\alpha} = (\lambda)$ and zero otherwise.
For the case $t=2$, we let $c(\lambda;\ul{\alpha}) = c\bigl(\lambda;(\alpha^1,\alpha^2)\bigr) = c^\lambda_{\alpha^1,\alpha^2}$.
For $t > 2$, we define $c(\lambda;\ul{\alpha})$ by induction on $t$ by setting
\begin{equation}\label{lrr_coeff_take_one_off:eq}
c\bigl(\lambda; \ul{\alpha}\bigr) =
\sum_{\beta \,\vdash\, n - |\alpha^1|}c^\lambda_{\alpha^1\!,\beta}\:c\bigl(\beta;(\alpha^2,\ldots,\alpha^t)\bigr).
\end{equation}

The \textit{Littlewood-Richardson rule} \cite[Theorem A1.3.3]{STANEC} states that  $c^\lambda_{\alpha,\beta}$ is equal to the number of skew
semistandard tableaux of shape $\lambda \setminus \alpha$ and type $\beta$ where the sequence obtained by concatenating its reversed
rows is a \textit{lattice word} (if $|\beta| \neq |\lambda|-|\alpha|$, then this number is zero). Here a lattice word is a finite
sequence of integers, allowing repetitions, such that if for any $r\geq0$ and any $i\geq0$ we let $\#^i_r$ be the number of times $i$
appears in the first $r$ places of the sequence, then for each $r$ we have $\#^1_r \geqslant \#^2_r \geqslant \#^3_r \geqslant \cdots$.
In particular, every Littlewood-Richardson coefficient is in fact a non-negative integer.

\subsection{Wreath products}
We shall now review some key definitions and constructions connected to wreath products. For more details, see \cite{CHTAN} and
\cite[chapter 4]{JAMKER}.
Let $n$ and $m$ be non-negative integers. The \textit{wreath product} of $S_n$ on $S_m$ is the group whose underlying set is the
Cartesian product of $S_n$ with $n$ copies of $S_m$. We shall write elements of $S_m{\wr}S_n$ as
$(\sigma;\alpha_1,\alpha_2,\ldots,\alpha_n)$ for $\alpha_1,\alpha_2,\ldots,\alpha_n \in S_m$ and $\sigma \in S_n$. Multiplication is given
by the formula
\begin{multline*}
	(\sigma;\alpha_1,\alpha_2,\ldots,\alpha_n)(\pi;\beta_1,\beta_2,\ldots,\beta_n) = \\
	\Bigl(\sigma\pi;\,(\alpha_{(1)\pi^{-1}}\beta_1),\,(\alpha_{(2)\pi^{-1}}\beta_2),\,\ldots\,,(\alpha_{(n)\pi^{-1}}\beta_n)\Bigr).
\end{multline*}
If $G$ is a subgroup of $S_m$ and $H$ a subgroup of $S_n$, we shall write $G{\wr}H$ for the subgroup of $S_m{\wr}S_n$
consisting of all elements $(\sigma;\alpha_1,\alpha_2,\ldots,\alpha_n)$ for $\alpha_1,\alpha_2,\ldots,\alpha_n \in G$ and $\sigma \in H$.
We shall make frequent use of such groups where $G$ and $H$ are each either the full symmetric group or a Young subgroup thereof, and we
shall often restrict or induce modules between such groups, for example
$\left.X\right\uparrow^{k(S_m{\wr}S_n)}_{k(S_m{\wr}S_\gamma)}$ and $\left.Y\right\downarrow^{k({S_m{\wr}S_n})}_{k({S_m{\wr}S_\gamma})}$
where $\gamma$ is some composition of $n$. As with the symmetric group, we shall de-clutter such expressions where possible by suppressing
the field and replacing the full symbols for subgroups of $S_m$ and $S_n$ with the subscript used to identify them,
so for example the above would be abbreviated to
$\left.X\right\uparrow^{m{\wr}n}_{m{\wr}\gamma}$ and $\left.Y\right\downarrow^{m{\wr}n}_{m{\wr}\gamma}$.

We now extend the notion of a Young subgroup to encompass multicompositions.
Let $\ul{\gamma} = (\gamma^1,\ldots,\gamma^t)$ be a $t$-multicomposition of $n$ ($t$ some non-negative integer) and let $\hat\gamma$ be
the composition of $n$ obtained by concatenating the compositions $\gamma^1,\ldots,\gamma^t$ in that order (so $\hat\gamma$ consists of
the parts of $\gamma^1$, followed by the parts of $\gamma^2$, and so on). We define the \textit{Young subgroup} of $S_n$ associated to
$\ul{\gamma}$ to be the Young subgroup $S_{\hat\gamma}$ associated to $\hat\gamma$, and we write $S_{\ul{\gamma}}$ for this subgroup. Thus we
have a canonical isomorphism $S_{\ul{\gamma}} \cong S_{\gamma^1} \times S_{\gamma^2} \times \cdots \times S_{\gamma^t}$.
Further, we note that $S_{\ul{\gamma}}$ is a subgroup of $S_{|\ul{\gamma}|}$.

We now recall several standard methods for constructing modules for wreath products, as described in
\cite[section 4.3]{JAMKER} and \cite[section 3]{CHTAN}. Recall that we are using \emph{right} modules.

Firstly, let $G$ be a subgroup of $S_m$, and
let $X$ be a $kG$-module. We define $X^{\widetilde\boxtimes n}$ to be the $k(G{\wr}S_n)$-module obtained by
equipping the $k$-vector space $X^{\otimes n}$ (that is, the tensor product over $k$ of $n$ copies of $X$) with the action given by the
formula
\[
	(x_1\otimes\cdots\otimes x_n)(\sigma;\alpha_1,\ldots,\alpha_n) = (x_{(1)\sigma^{-1}}\alpha_1)\otimes\cdots
	\otimes(x_{(n)\sigma^{-1}}\alpha_n)
\]
for $x_1,\ldots,x_n \in X$, $\alpha_1,\ldots,\alpha_n \in G$, $\sigma \in S_n$.
More generally, let $X_1,\,\ldots\,,X_t$ be $kG$-modules, and $\gamma = (\gamma_1,\,\ldots\,,\gamma_t)$ a composition of $n$ of length $t$.
We form a $k(G{\wr}S_{\gamma})$-module by equipping the $k$-vector space
$\bigl(X^{\otimes \gamma_1}_1\bigr)\otimes\bigl(X^{\otimes \gamma_2}_2\bigr)\otimes\,\cdots\,\otimes\bigl(X^{\otimes \gamma_t}_t\bigr)$
with the action given by the formula
\[
	(x_1\otimes\cdots\otimes x_n)(\sigma;\alpha_1,\ldots,\alpha_n) = (x_{(1)\sigma^{-1}}\alpha_1)\otimes\cdots
	\otimes(x_{(n)\sigma^{-1}}\alpha_n)
\]
where each $x_i$ lies in the appropriate $X_j$, $\alpha_1,\ldots,\alpha_n \in G$, and $\sigma \in S_{\gamma}$. We denote this module by
$\bigl(X_1,\,\ldots\,,X_t\bigr)^{\widetilde\boxtimes \gamma}$, and we note that $X^{\widetilde\boxtimes n}$ is the
special case of this construction where $\gamma$ has an $n$ in one place and all the other parts are 0.

Now let $G$ be a subgroup of $S_m$, $H$ be a subgroup of $S_n$, and $Y$ a $kH$-module. It is easy to check that we may make $Y$ into a
$k(G{\wr}H)$-module via the formula
\begin{equation}\label{y_inf_act:eq}
y(\sigma;\alpha_1,\ldots,\alpha_n) = y\sigma
\end{equation}
for $y \in Y$, $\alpha_1,\ldots,\alpha_n \in G$, and $\sigma \in H$. This module may be understood by noting that $G{\wr}H$ is the
semidirect product of the normal subgroup consisting of all elements $(e;\alpha_1,\ldots,\alpha_n)$ for $\alpha_1,\ldots,\alpha_n \in G$
with the subgroup consisting of all elements $(\sigma;e,\ldots,e)$ for $\sigma \in H$. This latter subgroup is canonically isomorphic to
$H$, and hence we see that the module obtained from $Y$ via \eqref{y_inf_act:eq} is the \emph{inflation} of $Y$ from $H$ to $G{\wr}H$
with respect to the semidirect product
structure. Hence we shall denote this module by $\mathrm{Inf}^{G{\wr}H}_H Y$.
Now let $H$ be a subgroup of $S_n$, $G$ be a subgroup of $S_m$, $Y$ be a $kH$-module, and further let $Z$ be a $k(G{\wr}H)$-module. Then
we define a $k(G{\wr}H)$-module $Z{\oslash}Y$ as follows: the underlying $k$-vector space is $Z{\otimes}Y$, and
the action is given by the formula
\[(z \otimes y)(\sigma;\alpha_1,\ldots,\alpha_n) = (z(\sigma;\alpha_1,\ldots,\alpha_n))\otimes(y\sigma)\]
for $z \in Z$, $y \in Y$, $\alpha_1,\ldots,\alpha_n \in G$, $\sigma \in H$.
Thus we see that we have an equality of $k(G{\wr}H)$-modules $Z{\oslash}Y = Z \otimes \mathrm{Inf}^{G{\wr}H}_H Y$
where the module on the right-hand side is the internal tensor product of the $k(G{\wr}H)$-modules $Z$ and $\mathrm{Inf}^{G{\wr}H}_H Y$.
Since taking the (internal) tensor product of group modules and inflating group modules are both exact functors, it follows that the
operation $- \oslash -$ preserves filtrations in both places, in the sense that a filtration of $Z$ by modules $X_i$ induces a filtration
of $Z \oslash Y$ by modules $X_i \oslash Y$ with the same multiplicities, and similarly for a filtration of $Y$.

We can combine the above constructions as follows: if $G$ is a subgroup of $S_m$, $X_1,\,\ldots\,,X_t$ are $kG$-modules and $Y$ is a
$kS_{\gamma}$-module for $\gamma$ a composition of $n$, then we obtain a $k(G{\wr}S_{\gamma})$-module
$\bigl(X_1,\,\ldots\,,X_t\bigr)^{\widetilde\boxtimes\gamma}{\oslash}Y$
with underlying vector space
$\bigl(X^{\otimes \gamma_1}_1\bigr)\otimes\bigl(X^{\otimes \gamma_2}_2\bigr)\otimes\,\cdots\,\otimes\bigl(X^{\otimes \gamma_t}_t\bigr){\otimes}Y$
and action given by the formula
\begin{multline}\label{xs_twid_osl_y_act:eq}
	(x_1{\otimes}\cdots{\otimes}x_n{\otimes}y)(\sigma;\alpha_1,\ldots,\alpha_n) = \\
	(x_{(1)\sigma^{-1}}\alpha_1)\otimes\cdots\otimes(x_{(n)\sigma^{-1}}\alpha_n){\otimes}(y\sigma)
\end{multline}
for $x_i \in X$, $\alpha_i \in G$, $y \in Y$, $\sigma \in S_{\gamma}$.

We now recall an elementary construction for producing $kS_{\gamma}$-modules $Y$ for use in the above constructions. Indeed, for each
$i \in \{1,\ldots,t\}$, let $Y_i$ be a right $kS_{\gamma_i}$-module. Now recall that we have a canonical identification of the group
$S_{\gamma}$ with the direct product $S_{\gamma_1}\times S_{\gamma_2}\times\cdots\times S_{\gamma_t}$ of groups. Thus any module for
$k(S_{\gamma_1}{\times}S_{\gamma_2}{\times}\cdots{\times}S_{\gamma_t})$ may be regarded as a $kS_{\gamma}$-module in a canonical way, and vice
versa. In particular, if $Y_i$ is a $kS_{\gamma_i}$-module for each $i$, then the external tensor product
$Y_1 \boxtimes Y_2 \boxtimes \cdots \boxtimes Y_t$, which is a
$k\left(S_{\gamma_1}\times S_{\gamma_2}\times\cdots\times S_{\gamma_t}\right)$-module, may be regarded as a $kS_{\gamma}$-module.

Now if $G$ is a subgroup of $S_m$ and $\gamma$ is a composition of $n$, then we have an obvious isomorphism between $G{\wr}S_{\gamma}$ and 
$\left(G{\wr}S_{\gamma_1}\right)\times\left(G{\wr}S_{\gamma_2}\right)\times\cdots\times\left(G{\wr}S_{\gamma_t}\right)$, and hence we have a
canonical identification of algebras between $k(G{\wr}S_{\gamma})$ and
$k\left(G{\wr}S_{\gamma_1}\right)\otimes k\left(G{\wr}S_{\gamma_2}\right)\otimes\cdots \otimes k\left(G{\wr}S_{\gamma_t}\right)$.
With this identification, it is now easy to see that we have an isomorphism of modules
\begin{multline}\label{mod_wprd_etp_vrs:eq}
\bigl(X_1,\,\ldots\,,X_t\bigr)^{\widetilde\boxtimes\gamma} {\oslash}\bigl(Y_1 \boxtimes Y_2 \boxtimes \cdots\boxtimes Y_t\bigr)
\cong\\
	\bigl(X_1^{\widetilde\boxtimes \gamma_1}\oslash Y_1\bigr)\boxtimes
	\bigl(X_2^{\widetilde\boxtimes \gamma_2}\oslash Y_2\bigr)\boxtimes\cdots
	\boxtimes\bigl(X_t^{\widetilde\boxtimes \gamma_t}\oslash Y_t\bigr)
\end{multline}
(this isomorphism was given in \cite[Lemma 3.2 (1)]{CHTAN}).

\begin{proposition}\label{res_comm_tilde_boxtimes:prop}
Let $G_1 \subseteq G_2$ be subgroups of $S_m$ and $X$ a $kG_2$-module. Then we have an isomorphism of $k(G_1 \wr S_n)$-modules
\[
\Bigl[X^{\widetilde\boxtimes n}\Bigr]\Bigr\downarrow^{G_2 \wr S_n}_{G_1 \wr S_n} \:\cong\:
\Bigl[X\bigr\downarrow^{G_2}_{G_1}\Bigl]^{\widetilde\boxtimes n}
\]
\end{proposition}
\begin{proof}
This is immediate from the definition of $(-)^{\widetilde\boxtimes n}$.
\end{proof}

\begin{proposition}\label{wreath_ind_res:prop}\cite[Lemma 3.2]{CHTAN}
Let $G$ be a subgroup of $S_m$.
Let $\alpha = (\alpha_1,\ldots,\alpha_t)$ be a composition of $n$ and let $V$ be a $k(G{\wr}S_n)$-module, $W$ be a
$k(G{\wr}S_\alpha)$-module, $X$ be a $kS_n$-module and $Y$ be a $kS_\alpha$-module. Then we have module isomorphisms
\begin{enumerate}
\item $\left.\bigl[V \oslash X\bigr]\right\downarrow^{G{\wr}n}_{G{\wr}\alpha} \cong
\bigl(\left.V\right\downarrow^{G{\wr}n}_{G{\wr}\alpha}\bigr) \oslash \left(\left.X\right\downarrow^{n}_{\alpha}\right)$
\item $V \oslash \left(\left.Y\right\uparrow^{n}_{\alpha}\right)
\cong \left.\bigl[\bigl(\left.V\right\downarrow^{G{\wr}n}_{G{\wr}\alpha}\bigr) \oslash Y\bigr]\right\uparrow^{G{\wr}n}_{G{\wr}\alpha}$
\item $\bigl(\left.W\right\uparrow^{G{\wr}n}_{G{\wr}\alpha}\bigr) \oslash X
\cong \left.\bigl[W \oslash \left(\left.X\right\downarrow^{n}_{\alpha}\right)\bigr]\right\uparrow^{G{\wr}n}_{G{\wr}\alpha}$
\end{enumerate}
where the symbols $n$ and $\alpha$ represent the subgroups $S_n$ and $S_\alpha$ of $S_n$, respectively.
\end{proposition}

\section{Wreath product Specht modules}
We now define analogues for the wreath product $S_m \wr S_n$ of the Specht modules of the symmetric group using the above constructions.
As mentioned in the introduction, although these constructions are well-known, the author is not aware that these modules have
previously been considered as analogues of the symmetric group Specht modules.

Firstly we define some useful notation. If $Y_1,\ldots,Y_s$ are $kS_m$-modules and $\ul{\eta} = (\eta^1,\ldots,\eta^s)$ is an
$s$-component multipartition of $n$, then we define the $k(S_m{\wr}S_n)$-module $S^{\ul{\eta}}(Y_1,\ldots,Y_s)$ by setting
\[
S^{\ul{\eta}}(Y_1,\ldots,Y_s)\;=\; \left.\left[\bigl(Y_1,\ldots,Y_s\bigr)^{\widetilde\boxtimes |\ul{\eta}|}{\oslash}
\bigl(S^{\eta^1}\boxtimes\cdots\boxtimes S^{\eta^s}\bigr)\right]\right\uparrow^{m{\wr}n}_{m{\wr}|\ul{\eta}|}.
\]

We take $r$ to be the number of distinct partitions of $m$, and we enumerate them in the lexicographic order as follows
\[ (m) = \mu^1 > \mu^2 > \,\ldots\, > \mu^r = (1^m). \]
Then for $\ul{\nu}=(\nu^1,\ldots,\nu^r)$ an $r$-multipartition of $n$, we define a $k(S_m{\wr}S_n)$-module
\[
	S^{\ul{\nu}} =S^{\ul{\nu}}(S^{\mu^1},\ldots,S^{\mu^r})
\]
and we call $S^{\ul{\nu}}$ the the \textit{Specht module} for $S_m{\wr}S_n$ associated to $\ul{\nu}$. For later convenience, we also
define a $k(S_m \wr S_{|\ul{\nu}|})$-module
\[
	T^{\ul{\nu}} = \bigl(S^{\mu^1},\ldots,S^{\mu^r}\bigr)^{\widetilde\boxtimes |\ul{\nu}|}{\oslash}
	\bigl(S^{\nu^1}\boxtimes\cdots\boxtimes S^{\nu^r}\bigr)
\]
so that $S^{\ul{\nu}} = T^{\ul{\nu}}\bigr\uparrow^{m{\wr}n}_{m{\wr}|\ul{\nu}|}$.
As mentioned above, the use of the name ``Specht module'' here is justified by the fact that these modules have the same relationship
with the simple modules of $k(S_m \wr S_n)$ as the Specht modules of $kS_n$ have with the simple of $kS_n$. This may be demonstrated by
noting that the wreath product Specht modules occur as the cell modules of a cellular structure on $k(S_m \wr S_n)$, in the sense of
Graham and Lehrer \cite{GLCA}. For details, see \cite{GCSWPA} and \cite[section 5.5]{PHDTHES}.

We now consider how filtrations of the $kS_m$-modules $Y_1,\ldots,Y_s$ induce filtrations of the $k(S_m \wr S_n)$-module
$S^{\ul{\eta}}(Y_1,\ldots,Y_s)$. This question was answered by Chuang and Tan
in \cite{CHTAN}, and the results we now present are taken from there. However, we shall present these results in a very slightly modified
form, using the notion of a \textit{multipartition matrix}, which is simply a matrix whose entries are multipartitions. We shall typically
denote the multipartition matrix whose $(i,j)$\tss{th} entry is the multipartition $\ul{\epsilon}^{ij}$ as $[\ul{\epsilon}]$. Thus a
multipartition matrix is simply a matrix whose entries are tuples of tuples of integers. Now let $s$ and $t$ be positive integers,
let $\alpha,\beta$ be compositions of the same integer $n$ and with lengths $s$ and $t$ respectively, and let $L$ be an $s{\times}t$
matrix with non-negative integer entries. We define $\mathrm{Mat}_{\ul{\Lambda}}(L;\alpha{\times}\beta)$ to be the
set of all $s{\times}t$ multipartition matrices $[\ul{\epsilon}]$ such that:
\begin{enumerate}
\item for each $i=1,\ldots,s$, the sum of all of the integers occurring in the $i$\tss{th} row of $[\ul{\epsilon}]$ is equal to the
$i$\tss{th} part of $\alpha$;
\item for each $j=1,\ldots,t$, the sum of all of the integers occurring in the $j$\tss{th} column of $[\ul{\epsilon}]$ is equal to the
$j$\tss{th} part of $\beta$;
\item the length of the $(i,j)$\tss{th} entry of $[\ul{\epsilon}]$ is equal to the $(i,j)$\tss{th} entry of $L$.
\end{enumerate}

From \cite{CHTAN} we have the following result. Note that \cite{CHTAN} formally makes the assumption that the modules $X_1,\ldots,X_t$
are pairwise non-isomorphic, but this is not in fact needed for the proof. For a very detailed proof of the result in this form, see
also section 6.4 of the author's PhD thesis, \cite{PHDTHES}.

\begin{proposition}\label{S_eta_Y_filt_S_nu_X:prop}\cite[Lemma 4.4, (1)]{CHTAN} (see also \cite[Proposition 6.4.1]{PHDTHES})
Let $Y_1,\ldots,Y_s$ and $X_1,\ldots,X_t$ be $kS_m$-modules such that for each $i=1,\ldots,s$ we have a filtration of $Y_i$ by
$X_1,\ldots,X_t$ where $X_j$ has multiplicity $a^i_j$. Let $\ul{\eta}$ be an $s$-component multipartition of $n$. Then
$S^{\ul{\eta}}(Y_1,\ldots,Y_s)$ has a filtration by the modules $S^{\ul{\nu}}(X_1,\ldots,X_t)$ for $\ul{\nu}$ a $t$-multipartitions of $n$,
where $S^{\ul{\nu}}(X_1,\ldots,X_t)$ has multiplicity
\[
\sum_{[\ul{\epsilon}] \in \mathrm{Mat}_{\ul{\Lambda}}(A;|\ul{\eta}|\times|\ul{\nu}|)}
\!\!\left(\prod^s_{i=1}c(\eta^i;R_i[\ul{\epsilon}])\right)\!\!\!\left(\prod^t_{j=1}c(\nu^j;C_j[\ul{\epsilon}])\right)
\]
where we define $A$ to be the $s \times t$ integer matrix whose $(i,j)$\tss{th} entry is $a^i_j$. Further, suppose that we have $s = t$
and moreover that we have $w_i = i$ for each $i=1,\ldots,t$. Then the module occurring at the bottom of this filtration is
$S^{\ul{\eta}}(X_1,\ldots,X_t)$.
\end{proposition}

\section{First Specht branching rule for wreath products}
For $m>0$, we can embed $S_{m-1} \wr S_n$ into $S_m \wr S_n$ using the canonical embedding of $S_{m-1}$ into $S_m$, thus identifying
$S_{m-1} \wr S_n$ with the subgroup of $S_m \wr S_n$ consisting of all elements $(\sigma;\alpha_1,\ldots,\alpha_n)$ where $\sigma \in S_n$
and each $\alpha_i$ is an element of the subgroup $S_{m-1}$ of $S_m$. Hence for $\ul{\lambda} = (\lambda^1,\ldots,\lambda^r)$ an
$r$-multipartition of $n$, we can consider the $k(S_{m-1} \wr S_n)$-module
\[ \left.S^{\ul{\lambda}}\right\downarrow^{m \wr n}_{(m-1) \wr n} \cong
T^{\ul{\lambda}}\bigl\uparrow^{m{\wr}n}_{m{\wr}|\ul{\lambda}|}\bigl\downarrow^{m \wr n}_{(m-1) \wr n}\]
obtained by restricting $S^{\ul{\lambda}}$ from $k(S_m \wr S_n)$ to $k(S_{m-1} \wr S_n)$. By  Mackey's Theorem, we have
\[
T^{\ul{\lambda}}\bigl\uparrow^{m{\wr}n}_{m{\wr}|\ul{\lambda}|}\bigl\downarrow^{m \wr n}_{(m-1) \wr n}\cong
\bigoplus_{u \in \mathcal{U}}\bigl(T^{\ul{\lambda}}\bigr)^u\bigl\downarrow^{(m{\wr}|\ul{\lambda}|)^u}_{(m{\wr}|\ul{\lambda}|)^u \,\cap\, (m-1) \wr n}
\bigl\uparrow^{(m-1) \wr n}_{(m{\wr}|\ul{\lambda}|)^u \,\cap\, (m-1) \wr n}
\]
where $\mathcal{U}$ represents a complete
non-redundant system of $(S_m \wr S_{|\ul{\lambda}|},S_{(m-1)} \wr S_n)$-double coset representatives in $S_m \wr S_n$, and where we allow
ourselves a slight abuse of notation by writing $(m \wr |\ul{\lambda}|)^u$ to represent the subgroup $(S_m \wr S_{|\ul{\lambda}|})^u$
conjugate to $S_m \wr S_{|\ul{\lambda}|}$ by $u$, and  $(m \wr |\ul{\lambda}|)^u \,\cap\, (m-1) \wr n$ for the intersection of this subgroup
with $S_{(m-1)} \wr S_n$. But it turns out that in fact the group $S_m \wr S_n$ is a single
$(S_m \wr S_{|\ul{\lambda}|},S_{(m-1)} \wr S_n)$-double coset. Indeed, choosing $(\sigma;\alpha_1,\ldots,\alpha_n)
\in S_m \wr S_n$, we have equalities of double cosets
\begin{align*}
S_m{\wr}S_{|\ul{\lambda}|}&\,(\sigma;\alpha_1,\ldots,\alpha_n)\,S_{(m-1)}{\wr}S_n\\
& = S_m{\wr}S_{|\ul{\lambda}|}\,(e;\alpha_{(1)\sigma},\ldots,\alpha_{(n)\sigma})(e;e,\ldots,e)(\sigma;e,\ldots,e)\,S_{(m-1)}{\wr}S_n\\
& = S_m{\wr}S_{|\ul{\lambda}|}\,(e;e,\ldots,e)\,S_{(m-1)}{\wr}S_n
\end{align*}
and so we may take $\mathcal{U} = \{(e;e,\ldots,e)\}$. We thus have
\[
\left.S^{\ul{\lambda}}\right\downarrow^{m \wr n}_{(m-1) \wr n} \cong
T^{\ul{\lambda}}\bigl\downarrow^{m{\wr}|\ul{\lambda}|}_{m{\wr}|\ul{\lambda}| \,\cap\, (m-1) \wr n}
\bigl\uparrow^{(m-1) \wr n}_{m{\wr}|\ul{\lambda}| \,\cap\, (m-1) \wr n}
\]
and clearly $\bigl(S_m \wr S_{|\ul{\lambda}|}\bigr) \cap \bigl(S_{(m-1)} \wr S_n\bigr) = S_{(m-1)} \wr S_{|\ul{\lambda}|}$ (note that formally these
are subgroups of $S_m \wr S_n$, so that $S_{(m-1)} \wr S_{|\ul{\lambda}|}$ is the subgroup of $S_m \wr S_n$ consisting of all elements
$(\sigma;\alpha_1,\ldots,\alpha_n)$ for $\sigma \in S_{|\ul{\lambda}|}$ and $\alpha_i \in S_{(m-1)} \leqslant S_m$). Thus we have
\begin{align*}
\left.S^{\ul{\lambda}}\right\downarrow^{m \wr n}_{(m-1) \wr n}
&\cong T^{\ul{\lambda}}
\bigl\downarrow^{m{\wr}|\ul{\lambda}|}_{(m-1){\wr}|\ul{\lambda}|}
\bigl\uparrow^{(m-1) \wr n}_{(m-1){\wr}|\ul{\lambda}|}\\
&\cong
\left[\bigboxtimes^r_{i=1} \bigl(S^{\mu^i}\bigr)^{\widetilde\boxtimes |\lambda^i|}\oslash S^{\lambda^i}\right]
\biggr\downarrow^{m{\wr}|\ul{\lambda}|}_{(m-1){\wr}|\ul{\lambda}|}
\biggr\uparrow^{(m-1) \wr n}_{(m-1){\wr}|\ul{\lambda}|}\\
&\cong
\left.\left[
\bigboxtimes^r_{i=1} \left[\left.\bigl(S^{\mu^i}\bigr)^{\widetilde\boxtimes |\lambda^i|}
\right\downarrow^{m \wr |\lambda^i|}_{(m-1) \wr |\lambda^i|}\right]
\oslash S^{\lambda^i}\right]\right\uparrow^{(m-1) \wr n}_{(m-1){\wr}|\ul{\lambda}|}\\
&\hspace{5em}\text{(it is easy to prove this directly)}\\
&\cong
\left.\left[
\bigboxtimes^r_{i=1} \Bigl(\left.S^{\mu^i}\right\downarrow^{m}_{m-1}\Bigr)^{\widetilde\boxtimes |\lambda^i|}
\oslash S^{\lambda^i}\right]\right\uparrow^{(m-1) \wr n}_{(m-1){\wr}|\ul{\lambda}|}\\
&\hspace{5em}\text{(by Proposition \ref{res_comm_tilde_boxtimes:prop})}\\
&\cong
S^{\ul{\lambda}}\left(S^{\mu^1}\bigr\downarrow^m_{m-1},\ldots,S^{\mu^r}\bigr\downarrow^m_{m-1}\right)\\
&\hspace{5em}\text{(using the isomorphism \eqref{mod_wprd_etp_vrs:eq}).}
\end{align*}
Now let us fix the partitions of $m-1$ just as we have done for $m$. Indeed, let $t$ be the number of distinct partitions of $m-1$, and
let
\[ (m-1) = \theta^1 > \theta^2 > \,\ldots\, > \theta^t = (1^{m-1}) \]
be the partitions of $m-1$ in lexicographic order. Then by Theorem \ref{jam_branch_rule_sn:thm}, we have for any $i \in \{1,\ldots,r\}$
a filtration of $S^{\mu^i}\bigr\downarrow^{m}_{m-1}$ by the modules $S^{\theta^j}$, where $S^{\theta^j}$ has multiplicity $a^i_j$, where
we define $a^i_j$ to be 1 if $\theta^j$ can be obtained by removing a box from $\mu^i$, and zero otherwise.
It now follows by Proposition \ref{S_eta_Y_filt_S_nu_X:prop} that we have a filtration of
$S^{\ul{\lambda}}\bigr\downarrow^{m \wr n}_{(m-1) \wr n}$ by the modules $S^{\ul{\nu}}$ for $\ul{\nu}$ a $t$-multipartition of $n$ where
$S^{\ul{\nu}}$ has multiplicity
\[
\sum_{[\ul{\epsilon}] \,\in\, \mathrm{Mat}_{\ul{\Lambda}}(A;|\ul{\lambda}|\times|\ul{\nu}|)}
\left(\prod^r_{i=1}c(\lambda^i;R_i[\ul{\epsilon}])\:\cDot\:\prod^t_{j=1}c(\nu^j;C_j[\ul{\epsilon}])\right)
\]
where $A$ is the $r \times t$ integer matrix whose $(i,j)$\tss{th} entry is $a^i_j$. This filtration is the basis of our desired Specht
branching rule, but we would like some kind of combinatorial interpretation of the multiplicities which occur. Our task
is now to find such an interpretation.

So with $\ul{\lambda}$ as above and $\ul{\nu}$ a $t$-multipartition of $n$, consider, for a given multipartition
matrix
$[\ul{\epsilon}] \in \mathrm{Mat}_{\ul{\Lambda}}(A;|\ul{\lambda}|\times|\ul{\nu}|)$ the coefficient
\begin{equation}\label{mpt_mat_bran_coeff:eq}
\prod^r_{i=1}c(\lambda^i;R_i[\ul{\epsilon}])\:\cDot\:\prod^t_{j=1}c(\nu^j;C_j[\ul{\epsilon}]).
\end{equation}
Now the
$(i,j)$\tss{th} entry of $[\ul{\epsilon}]$ is a multipartition of length 1, say $(\epsilon^{ij})$, if $\theta^j$ can be obtained by
removing a box from $\mu^i$, and $()$ otherwise. This gives us an alternative way to think of such multipartition matrices and calculate
the associated coefficient \eqref{mpt_mat_bran_coeff:eq}, as we shall now explain.

Recall that we can arrange the set of all partitions of all non-negative integers in a graphical structure called the
\textit{Young graph}, by arranging the partitions in layers, with the partitions of size $s$ forming the $s$\tss{th} layer, and then for
each partition $\lambda \vdash s$ in the $s$\tss{th} layer, drawing an edge from $\lambda$ to each partition of $s-1$ in the
$(s-1)$\tss{th} layer which can be obtained from $\lambda$ by removing a single box. For example, the second and third rows of the Young
graph, together with the edges connecting them, look like this
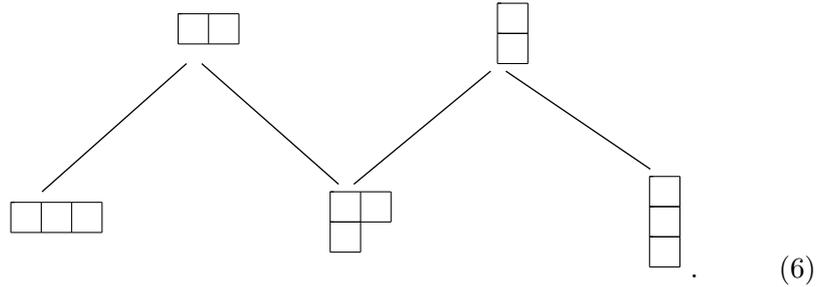
\begin{equation}\label{snd_thd_layer_yng_grph:eq}
\begin{tikzpicture}[line width=0.5pt]
	\node at (2,2.5){\mytab[4mm]{{{,}}}};
	\node at (6,2.5){\mytab[4mm]{{{\,},{\,}}}};
	\node at (0,0){\mytab[4mm]{{{,,}}}};
	\node at (4,0){\mytab[4mm]{{{,},{\,}}}};
	\node at (8,0){\mytab[4mm]{{{\,},{\,},{\,}}}};
	\draw (1.9,2.1) to (0,0.4);
	\draw (2.1,2.1) to (3.9,0.5);
	\draw (5.9,2) to (4.1,0.5);
	\draw (6.1,2) to (8,0.7);
\end{tikzpicture}.
\end{equation}
For our purposes, we are interested in the subgraph of the Young graph consisting of the $m$\tss{th} and $(m-1)$\tss{th} layers together
with the edges connecting them. Let us call this subgraph $\mathcal{Y}_m$. So for example if $m=3$, $\mathcal{Y}_3$ is the graph
\eqref{snd_thd_layer_yng_grph:eq}. We see that there is a natural one-to-one correspondence between the 1's in the matrix $A$ and the
edges in $\mathcal{Y}_m$. Indeed, a 1 in the $(i,j)$\tss{th} place of $A$ corresponds to an edge linking $\theta^j \vdash m-1$ and
$\mu^i \vdash m$ in $\mathcal{Y}_m$. We now see that a multipartition matrix
$[\ul{\epsilon}] \in \mathrm{Mat}_{\ul{\Lambda}}(A;|\ul{\lambda}|\times|\ul{\nu}|)$ may be identified with a labelling of the
edges in $\mathcal{Y}_m$ by partitions. Indeed, to obtain such a labelling from such a matrix $[\ul{\epsilon}]$, we label the edge linking
$\theta^j$ and $\mu^i$ in $\mathcal{Y}_m$, if it exists, with the partition $\epsilon^{ij}$ which is the unique entry of the length 1
multipartition which is the $(i,j)$\tss{th} entry of $[\ul{\epsilon}]$. We may easily see that we have now established a one-to-one
correspondence between on the one hand the set $\mathrm{Mat}_{\ul{\Lambda}}(A;|\ul{\lambda}|\times|\ul{\nu}|)$ and on the other hand
labellings of the edges of $\mathcal{Y}_m$ by integer partitions, such that for each $i = 1,\ldots,r$ the sizes of the partitions
labelling the edges touching the node $\mu^i \vdash m$ of $\mathcal{Y}_m$ add up to $|\lambda^i|$, and similarly for each $j = 1,\ldots,t$
the sizes of the partitions labelling the edges touching the node $\theta^j \vdash m-1$ of $\mathcal{Y}_m$ add up to $|\nu^i|$. We shall
henceforth call such a labelling of $\mathcal{Y}_m$ a \textit{labelling of shape $|\ul{\lambda}|\times|\ul{\nu}|$}. The diagram
\eqref{yng_grph_gd_lab_eg_lambd_nu:eq} below is an example of such a labelling.

We now explain how to calculate the coefficient \eqref{mpt_mat_bran_coeff:eq} associated to a labelling of $\mathcal{Y}_m$ of shape
$|\ul{\lambda}|\times|\ul{\nu}|$. In order to do this, we need to introduce a graph which is a modified version of $\mathcal{Y}_m$.
Indeed, recall that we have multipartitions $\ul{\lambda} = (\lambda^1,\ldots,\lambda^r)$ and $\ul{\nu} = (\nu^1,\ldots,\nu^t)$ of $n$.
We define $\mathcal{Y}_m(\ul{\lambda},\ul{\nu})$ to be the graph obtained by replacing each partition $\mu^i \vdash m$ with $\lambda^i$,
and each partition $\theta^j \vdash m-1$ with $\nu^j$. Thus for example if $m=3$ (so that $r=3$ and $t=2$) and $n=6$, and we take
$\ul{\lambda} = \bigl((2),(1,1),(1,1)\bigr)$ and $\ul{\nu} = \bigl((3),(2,1)\bigr)$, then $\mathcal{Y}_3(\ul{\lambda},\ul{\nu})$ is the
graph
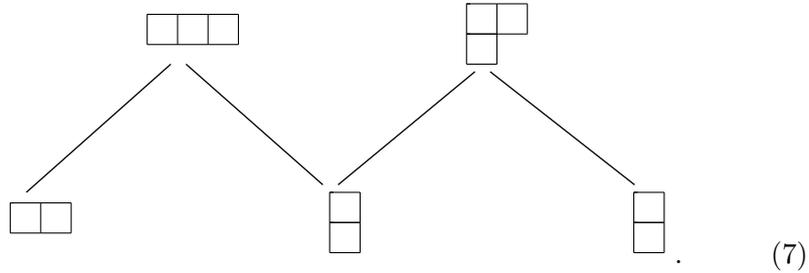
\begin{equation}\label{yng_grph_eg_lambd_nu:eq}
\begin{tikzpicture}[line width=0.5pt]
	\node at (2,2.5){\mytab[4mm]{{{,,}}}};
	\node at (6,2.5){\mytab[4mm]{{{,},{\,}}}};
	\node at (0,0){\mytab[4mm]{{{,}}}};
	\node at (4,0){\mytab[4mm]{{{\,},{\,}}}};
	\node at (8,0){\mytab[4mm]{{{\,},{\,}}}};
	\draw (1.9,2.1) to (0,0.4);
	\draw (2.1,2.1) to (3.9,0.5);
	\draw (5.9,2) to (4.1,0.5);
	\draw (6.1,2) to (8,0.5);
\end{tikzpicture}.
\end{equation}
We now see that a labelling of $\mathcal{Y}_m$ of shape $|\ul{\lambda}|\times|\ul{\nu}|$ corresponds to a labelling of the edges
$\mathcal{Y}_m(\ul{\lambda},\ul{\nu})$ by partitions in such a way that, for each partition $\gamma$ lying at a node of
$\mathcal{Y}_m(\ul{\lambda},\ul{\nu})$, the sizes of the partitions labelling all the edges touching $\gamma$ add up to $|\gamma|$. We
call such a labelling of $\mathcal{Y}_m(\ul{\lambda},\ul{\nu})$ a \textit{good labelling} of $\mathcal{Y}_m(\ul{\lambda},\ul{\nu})$. To
continue our example, one good labelling of the graph $\mathcal{Y}_3(\ul{\lambda},\ul{\nu})$ depicted in \eqref{yng_grph_eg_lambd_nu:eq}
is
\tikzstyle{ldot}=[draw, fill =black, circle, inner sep=0pt, minimum size=2pt]
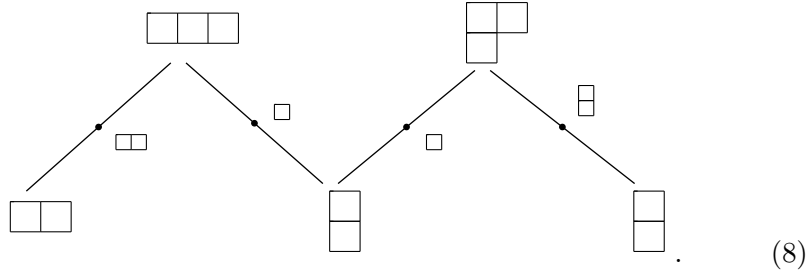
\begin{equation}\label{yng_grph_gd_lab_eg_lambd_nu:eq}
\begin{tikzpicture}[line width=0.5pt]
	\node at (2,2.5){\mytab[4mm]{{{,,}}}};
	\node at (6,2.5){\mytab[4mm]{{{,},{\,}}}};
	\node at (0,0){\mytab[4mm]{{{,}}}};
	\node at (4,0){\mytab[4mm]{{{\,},{\,}}}};
	\node at (8,0){\mytab[4mm]{{{\,},{\,}}}};
	\draw (1.9,2.1) to (0,0.4);
	\draw (2.1,2.1) to (3.9,0.5);
	\draw (5.9,2) to (4.1,0.5);
	\draw (6.1,2) to (8,0.5);
	\node[ldot] at (0.95,1.25){};
	\node[ldot] at (3,1.3){};
	\node[ldot] at (5,1.25){};
	\node[ldot] at (7.05,1.25){};
	\node at (1.2,1){\mytab[2mm]{{{,}}}};
	\node at (3.2,1.4){\mytab[2mm]{{{\,}}}};
	\node at (5.2,1){\mytab[2mm]{{{\,}}}};
	\node at (7.2,1.6){\mytab[2mm]{{{\,},{\,}}}};
\end{tikzpicture}.
\end{equation}
Looking back through our arguments, we see that this labelling corresponds to the multipartition matrix
\[
\bordermatrix{
      & (3)             &(2,1)              \cr
(2)   & \bigl((2)\bigr) & \bigl(\bigr)      \cr
(1,1) & \bigl((1)\bigr) & \bigl((1)\bigr)   \cr
(1,1) & \bigl(\bigr)    & \bigl((1,1)\bigr)
}
\]
(where we have labelled the rows and columns with the entries of $\ul{\lambda}$ and $\ul{\nu}$ respectively)
and further we see that the coefficient \eqref{mpt_mat_bran_coeff:eq} associated to this multipartition matrix is
\begin{multline*}
c\bigl((2);\bigl((2)\bigr)\bigr)\cdot c\bigl((1,1);\bigl((1),(1)\bigr)\bigr)\cdot c\bigl((1,1);\bigl((1,1)\bigr)\bigr)\cdot\\
c\bigl((3);\bigl((2),(1)\bigr)\bigr)\cdot c\bigl((2,1);\bigl((1),(1,1)\bigr)\bigr).
\end{multline*}
By using our definition of the Littlewood-Richardson coefficient $c(\lambda;\ul{\alpha})$ and the Littlewood-Richardson rule, we may see
that each of these Littlewood-Richardson coefficients is 1, and hence the coefficient
associated to the graph \eqref{yng_grph_gd_lab_eg_lambd_nu:eq} is 1.

In the general case, we see that the coefficient associated to a good labelling of $\mathcal{Y}_m(\ul{\lambda},\ul{\nu})$ is formed by
taking the product, over all partitions $\gamma$ which are nodes of $\mathcal{Y}_m(\ul{\lambda},\ul{\nu})$ (that is, over all partitions
of $m$ and of $m-1$), of the Littlewood-Richardson coefficients $c\bigl(\gamma;(\delta^1,\ldots,\delta^s)\bigr)$, where
$\delta^1,\ldots,\delta^s$ are the partitions
labelling all of the edges which touch $\gamma$ in $\mathcal{Y}_m(\ul{\lambda},\ul{\nu})$. If $\mathcal{L}$ is a good labelling of
$\mathcal{Y}_m(\ul{\lambda},\ul{\nu})$, we denote this coefficient by $\mathcal{M}\bigl(\mathcal{L}\bigr)$.

We have now proved the following Specht branching rule, and we note that the multiplicities in this theorem are independent of the field
$k$.

\begin{theorem}
Let $m>0$, and as above let $r$ be the number of distinct partitions of $m$ and $t$ the number of distinct partitions of $m-1$.
Let $\ul{\lambda}$ be an $r$-multipartition of $n$. Then we have a filtration of the $k(S_{m-1} \wr S_n)$-module
$\left.S^{\ul{\lambda}}\right\downarrow^{m \wr n}_{(m-1) \wr n}$ by Specht modules $S^{\ul{\nu}}$ for $t$-multipartitions $\ul{\nu}$ of $n$,
where the multiplicity of $S^{\ul{\nu}}$ is the sum over all good labellings $\mathcal{L}$ of $\mathcal{Y}_m(\ul{\lambda},\ul{\nu})$ of
the coefficients $\mathcal{M}\bigl(\mathcal{L}\bigr)$.
\end{theorem}

Let us now extend our example to calculate the multiplicity which $S^{((3),(2,1))}$ has in our filtration of
$S^{((2),(1,1),(1,1))}\bigr\downarrow^{3\,\wr\,6}_{2\,\wr\,6}$. We have already calculated that the coefficient
$\mathcal{M}\bigl(\mathcal{L}\bigr)$ is
equal to 1 when $\mathcal{L}$ is the labelling \eqref{yng_grph_gd_lab_eg_lambd_nu:eq}. We shall show that if
$\ul{\lambda} = \bigl((2),(1,1),(1,1)\bigr)$
and $\ul{\nu} = \bigl((3),(2,1)\bigr)$, then for any good labelling $\mathcal{L}$ of $\mathcal{Y}_3(\ul{\lambda},\ul{\nu})$ other than 
\eqref{yng_grph_gd_lab_eg_lambd_nu:eq}, we have $\mathcal{M}\bigl(\mathcal{L}\bigr) = 0$. Thus the multiplicity which we seek is in fact
1. Indeed, suppose that we have some good labelling $\mathcal{L}$ of $\mathcal{Y}_3(\ul{\lambda},\ul{\nu})$. Then $\mathcal{L}$ is equal
to
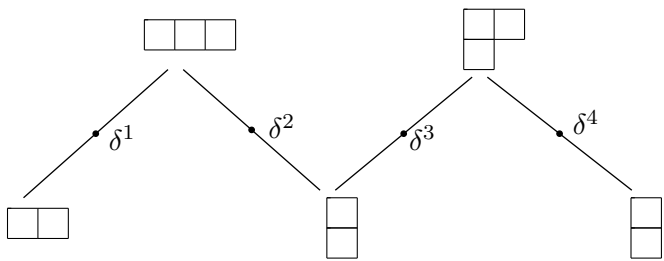
\[\begin{tikzpicture}[line width=0.5pt]
	\node at (2,2.5){\mytab[4mm]{{{,,}}}};
	\node at (6,2.5){\mytab[4mm]{{{,},{\,}}}};
	\node at (0,0){\mytab[4mm]{{{,}}}};
	\node at (4,0){\mytab[4mm]{{{\,},{\,}}}};
	\node at (8,0){\mytab[4mm]{{{\,},{\,}}}};
	\draw (1.9,2.1) to (0,0.4);
	\draw (2.1,2.1) to (3.9,0.5);
	\draw (5.9,2) to (4.1,0.5);
	\draw (6.1,2) to (8,0.5);
	\node[ldot] at (0.95,1.25){};
	\node[ldot] at (3,1.3){};
	\node[ldot] at (5,1.25){};
	\node[ldot] at (7.05,1.25){};
	\node at (1.3,1.2){$\delta^1$};
	\node at (3.4,1.35){$\delta^2$};
	\node at (5.25,1.2){$\delta^3$};
	\node at (7.4,1.4){$\delta^4$};
\end{tikzpicture}\]
for some integer partitions $\delta^1,\delta^2,\delta^3,\delta^4$. Now by the definition of a good labelling of
$\mathcal{Y}_3(\ul{\lambda},\ul{\nu})$, we see that we must have $|\delta^1|=2,|\delta^2|=1,|\delta^3|=1,|\delta^4|=2$, so that
$\delta^2 = \delta^3 = (1)$. We now see that
\begin{multline*}
\mathcal{M}\bigl(\mathcal{L}\bigr) =
c\bigl((2);\bigl(\delta^1\bigr)\bigr)\cdot c\bigl((1,1);\bigl((1),(1)\bigr)\bigr)\cdot c\bigl((1,1);\bigl(\delta^4\bigr)\bigr)\cdot\\
c\bigl((3);\bigl(\delta^1,(1)\bigr)\bigr)\cdot c\bigl((2,1);\bigl((1),\delta^4\bigr)\bigr).
\end{multline*}
By our definition of the Littlewood-Richardson coefficient $c(\lambda;\ul{\alpha})$, the only case where this is nonzero is the case
where $\delta^1 = (2)$ and $\delta^4 = (1,1)$, as in \eqref{yng_grph_gd_lab_eg_lambd_nu:eq}.

\section{Tableau combinatorics}
We now examine some tableau combinatorics which we shall use to help us understand the double cosets of certain pairs of subgroups in
$S_n$. The material in this section is taken from the account given by Wildon in his unpublished note \cite{WLDDCR}.

Throughout this section we fix $\alpha = (\alpha_1,\alpha_2,\ldots,\alpha_l)$ to be a composition of $n$ of length $l$, and
$\gamma = (\gamma_1,\gamma_2,\ldots,\gamma_t)$ to be a composition of $n$ of length $t$.
For our fixed composition $\alpha$ of $n$, we let $S_n$ act (from the right) on both the set of tableaux of shape $\alpha$ and type
$\gamma$, by permuting the entries of a tableau as follows.
For a tableau $\tau$, number the boxes of the tableau from 1 to $n$ going from left to right across each row in turn, starting with the
top row and working down. Let $\sigma \in S_n$. Then $\tau\sigma$ is defined to be the tableau obtained from $\tau$ by moving the number
in box number $i$ to box number $(i)\sigma$, for each $i=1,\ldots,n$. For example, let us take $n=13$,
$\alpha = (5,3,4,1)$, $\gamma = (4,5,4)$, $\sigma = (1,12,3,6)(5,7,13)(8,10) \in S_{13}$, and
\[
\tau \quad=\quad 
\mytab{
{{1,2,1,3,2},
{2,3,2},
{2,3,1,3},
{1}}
}\quad.
\]
The reader may verify that we have
\[
\tau\sigma \quad=\quad 
\mytab{
{{2,2,3,3,1},
{1,2,3},
{2,2,1,1},
{3}}
}\quad.\]

It is easy to see that this definition does indeed yield a $S_n$ action as claimed, and it is obvious that this $S_n$ action is
transitive. It is natural to ask what the stabilizer of a given tableau is under this action, and in order to answer this we now consider
certain special tableaux of shape $\alpha$ and type $\gamma$. Indeed, for our compositions $\alpha$ and $\gamma$, we
construct the \textit{standard tableau of shape $\alpha$ and type $\gamma$} as follows: we begin with a Young diagram of shape $\alpha$
with the boxes numbered as described above, and then working from box 1 to box $n$ we enter first $\gamma_1$ 1's, then $\gamma_2$ 2's,
and so on. We denote this tableau by $\tau^\alpha_{\gamma}$. For example, if we take
$n=13$, $\alpha = (2,0,3,1,3,4)$ and $\gamma=(3,5,0,4,1)$, then we have
\[
\tau^\alpha_{\gamma}\quad=\quad
\mytab{
{{1,1},
{},
{1,2,2},
{2},
{2,2,4},
{4,4,4,5}}
}.
\]

\begin{proposition}\label{tab_stab:prop}(See for example \cite{WLDDCR}, proof of Proposition 5.2)
For any $\sigma \in S_n$, we have $\mathrm{Stab}\bigl(\tau^\alpha_\gamma\sigma\bigr) = \bigl(S_\gamma\bigr)^\sigma$,
where we write $\mathrm{Stab}(-)$ to denote a stabilizer.
\end{proposition}
\begin{proof}
It is clear from the definition of $\tau^\alpha_\gamma$ that its stabilizers under the action of $S_n$ is
the Young subgroup $S_\gamma$. Now let $\sigma \in S_n$. Then for any $\theta \in S_n$ we have
that $\theta \in \mathrm{Stab}\left(\tau^\alpha_\gamma\sigma\right)$ if and only if $\tau^\alpha_\gamma\sigma = \tau^\alpha_\gamma\sigma\theta$,
which happens if an only if $\sigma\theta\sigma^{-1} \in \mathrm{Stab}\left(\tau^\alpha_\gamma\right) = S_\gamma$.
\end{proof}

Our purpose in studying tableaux is to gain an understanding of certain kinds of double cosets in $S_n$, and we shall next
show how we may use a particular subset of tableaux to index these double cosets in a natural way.
We say that a tableau of shape $\alpha$ and type $\gamma$ has \textit{weakly increasing rows} if the entries in its rows are weakly
increasing from left to right.

We now seek a condition on $\sigma \in S_n$ which ensures that the tableau $\tau^\alpha_{\gamma}\sigma$
has weakly increasing rows. To do this, we recall the notion of the \textit{length} of a permutation, which is
defined to be the total number of inversions of the permutation, where an \textit{inversion} of a permutation $\sigma \in S_n$ is a pair
$(i,j)$ such that $1 \leq i < j \leq n$ and $(i)\sigma > (j)\sigma$.

We shall prove that if $\sigma \in S_n$ is of minimal length in its $S_\alpha$-coset $\sigma S_\alpha$, then the tableau
$\tau^\alpha_{\gamma}\sigma$ has weakly increasing rows. For this, we shall need a well-known
combinatorial fact. We define a \textit{descent} of $\sigma$ to be an inversion $(j,j+1)$ of $\sigma$ for some $1 \leq j < n$.

\begin{lemma}\label{len_one_less:lem}
Let $\sigma \in S_n$, and suppose that $(j,j+1)$ is a descent of $\sigma^{-1}$. Then
$\mathrm{len}\bigl(\sigma (j,j+1)\bigr) = \mathrm{len}(\sigma) - 1$.
\end{lemma}
\begin{proof}(\cite{WLDDCR}, Lemma 2.1)
The claim will be established by proving two facts for any $\theta \in S_n$: firstly, that
$\mathrm{len}(\theta) = \mathrm{len}(\theta^{-1})$, and secondly that if $(j,j+1)$ is a descent of $\theta$, then
$\mathrm{len}\bigl((j,j+1)\theta\bigr) = \mathrm{len}(\theta) - 1$. Indeed, we then have
\begin{align*}
\mathrm{len}\bigl(\sigma (j,j+1)\bigr) &= \mathrm{len}\bigl(((j,j+1)\sigma^{-1})^{-1}\bigr)\\
&= \mathrm{len}\bigl((j,j+1)\sigma^{-1}\bigr)\\
&= \mathrm{len}\bigl(\sigma^{-1}\bigr) - 1\\
&= \mathrm{len}(\sigma) - 1.
\end{align*}

To see  that $\mathrm{len}(\theta) = \mathrm{len}(\theta^{-1})$, we note that it is easy to prove directly that $(x,y)$ is an inversion
of $\theta$ if and only if $((y)\theta,(x)\theta)$ is an inversion of $\theta^{-1}$. Further,  the map
$(x,y) \longmapsto ((y)\theta,(x)\theta)$ is clearly a bijection from $\{1,\ldots,n\}\times\{1,\ldots,n\}$ to itself. Hence
the inversions of $\theta$ and $\theta^{-1}$ are in bijection, so that $\mathrm{len}(\theta) = \mathrm{len}(\theta^{-1})$.

For the second property, we have trivially for any $x,y \in \{1,\ldots,n\}$ that $(x)\theta > (y)\theta$ if and only if
$(x)(j,j+1)(j,j+1)\theta > (y)(j,j+1)(j,j+1)\theta$.
From this we may easily see that if $x < y$ and the pair $(x,y)$ does not equal the pair $(j,j+1)$, then $(x)(j,j+1) < (y)(j,j+1)$.
Moreover, $(x,y)$ is then an inversion of $\theta$ if and only if $\Bigl((x)(j,j+1),\:(y)(j,j+1)\Bigr)$ is an inversion of
$(j,j+1)\theta$. Further, the pair $(j,j+1)$ is by assumption a descent of $\theta$ but is not a descent of $(j,j+1)\theta$, and the
second property is now established.
\end{proof}

\begin{proposition}\label{min_len_winc_rows:prop}(Compare \cite{WLDDCR}, Proposition 5.2 and Theorem 4.1)
If $\sigma \in S_n$ is of minimal length in its left $S_\alpha$-coset $\sigma S_\alpha$, then $\tau^\alpha_{\gamma}\sigma$ has
weakly increasing rows.
\end{proposition}
\begin{proof}
Suppose that $\tau^\alpha_{\gamma}\sigma$ does not have weakly increasing rows. Indeed, suppose that the $i$\textsuperscript{th} row of
$\tau^\alpha_{\gamma}\sigma$ is not weakly increasing, and let us define $a = 1 + \sum^{i-1}_{j=1} \alpha_i$ and $b = \sum^i_{j=1}\alpha_i$
so that (with our numbering of the boxes of a Young diagram as in the definition of the action of $S_n$) the boxes on the
$i$\textsuperscript{th} row of $\tau^\alpha_{\gamma}\sigma$ are numbered from $a$ to $b$. The fact that the $i$\tss{th} row of
$\tau^\alpha_{\gamma}\sigma$ is not weakly increasing means that we have some $(p,q)$ with $a \leq p < q \leq b$ such that the entry in the
box of $\tau^\alpha_{\gamma}\sigma$ with number $p$ is greater than the entry in the box of $\tau^\alpha_{\gamma}\sigma$ with number $q$. Now
by the definition of the action of $S_n$ on tableaux, we have for any $j$ that the entry which is in box number $j$ in
$\tau^\alpha_{\gamma}\sigma$ is the entry from box number $(j)\sigma^{-1}$ in $\tau^\alpha_{\gamma}$. By the definition of $\tau^\alpha_{\gamma}$,
if $i < j$ then the entry in the box of $\tau^\alpha$ with number $i$ is less than the entry in the box of $\tau^\alpha_{\gamma}$ with number
$j$. Hence we must have $(p)\sigma^{-1} > (q)\sigma^{-1}$, and so $(p,q)$ is an inversion of $\sigma^{-1}$. This implies that there must be
a \emph{descent} $(j,j+1)$ of $\sigma^{-1}$ such that $a\leq j < b$, for if not then we must have
\[ (a)\sigma^{-1} < (a+1)\sigma^{-1} < \cdots < (b-1)\sigma^{-1} < (b)\sigma^{-1}, \]
a contradiction. But then $\sigma(j, j+1) \in \sigma S_\alpha$ since $(j, j+1) \in S_\alpha$, and by
Lemma \ref{len_one_less:lem}, $\sigma(j, j+1)$ has length one less than $\sigma$, contradicting the minimality of the length of $\sigma$
in $\sigma S_\alpha$.
\end{proof}

We now demonstrate how tableaux with weakly increasing rows can be used to index double cosets. Let us define $\mathcal{W}^\alpha_{\gamma}$
to be the set of all tableaux of shape $\alpha$ and type $\gamma$ with weakly increasing rows.  Further, let us take
$\Omega^\alpha_{\gamma}$ to be a complete system of $(S_\gamma,S_\alpha)$-double coset representatives in $S_n$, where each element $\sigma$ of
$\Omega^\alpha_\gamma$ is of minimal length in its left coset $\sigma S_\alpha$.

\begin{proposition}\label{tab_dcos_bij:prop}(\cite{WLDDCR}, Corollary 5.1)
The map $\Omega^\alpha_\gamma \longrightarrow \mathcal{W}^\alpha_\gamma$, $\sigma \longmapsto \tau^\alpha_\gamma\sigma$ is a bijection.
\end{proposition}
\begin{proof}
To prove that the map  is onto, let $\tau$ be an element of $\mathcal{W}^\alpha_\gamma$. Then certainly
$\tau = \tau^\alpha_\gamma\theta$ for some $\theta \in S_n$, since our action of $S_n$ on tableaux is transitive.
But $\theta = u \sigma v$ for some $\sigma \in \Omega^\alpha_\gamma$, $u \in S_\gamma$, $v \in S_\alpha$, so that
$\tau = \tau^\alpha_\gamma u \sigma v$. Now by Proposition \ref{tab_stab:prop}, the stabilizer of $\tau^\alpha_\gamma$ under the action of
$S_n$ is $S_\gamma$, and so $\tau = \tau^\alpha_\gamma\sigma v$. Hence $\tau v^{-1} = \tau^\alpha_\gamma\sigma$. But $\sigma$ is certainly of
minimal length in its left $S_\alpha$-coset, and hence by Proposition \ref{min_len_winc_rows:prop} $\tau^\alpha_\gamma\sigma$ has weakly
increasing rows, so $\tau v^{-1}$ has weakly increasing rows. But $v^{-1} \in S_\alpha$, and so the action of $v^{-1}$ on $\tau$ just
permutes the elements within each row of $\tau$. The
fact that $\tau v^{-1}$ and $\tau$ both have weakly increasing rows now implies that $\tau = \tau v^{-1}$ and thus that
$\tau = \tau^\alpha_\gamma\sigma$.

To see that the map is one-to-one, suppose that $\tau^\alpha_\gamma\sigma_1 = \tau^\alpha_\gamma\sigma_2$ for $\sigma_1,\sigma_2 \in
\Omega^\alpha_\gamma$. Thus $\tau^\alpha_\gamma\sigma_1\sigma_2^{-1} = \tau^\alpha_\gamma$ and hence by Proposition \ref{tab_stab:prop}
$\sigma_1\sigma_2^{-1} \in S_\gamma$. It now follows at once that $S_\gamma \sigma_1 S_\alpha = S_\gamma \sigma_2 S_\alpha$ and hence that
$\sigma_1 = \sigma_2$.
\end{proof}

\begin{corollary}\label{tabs_give_cosets:cor}
Suppose that we have $\sigma_1,\ldots,\sigma_N \in S_n$ such that if $i \neq j$ then
$\tau^\alpha_\gamma\sigma_i \neq \tau^\alpha_\gamma\sigma_j$ and further $\{\tau^\alpha_\gamma\sigma_i \mid 1 \leqslant i \leqslant N\} =
\mathcal{W}^\alpha_\gamma$. Then $\sigma_1,\ldots,\sigma_N$ is a complete system of $(S_\gamma,S_\alpha)$-double coset representatives in
$S_n$ without redundancy.
\end{corollary}
\begin{proof}
With our system of $(S_\gamma,S_\alpha)$-double coset representatives $\Omega^\alpha_{\gamma}$ as above, we may by Proposition
\ref{tab_dcos_bij:prop} list the distinct elements of $\Omega^\alpha_{\gamma}$ as $\omega_1,\ldots,\omega_N$ such that
$\tau^\alpha_\gamma\sigma_i = \tau^\alpha_\gamma\omega_i$. This implies that $\tau^\alpha_\gamma = \tau^\alpha_\gamma\omega_i\sigma_i^{-1}$, and
hence that $\omega_i\sigma_i^{-1} \in \mathrm{Stab}(\tau^\alpha_\gamma)$, so that by Proposition \ref{tab_stab:prop} we have
$\omega_i\sigma_i^{-1} \in S_\gamma$. Hence $S_\gamma \sigma_i S_\alpha = S_\gamma (\omega_i\sigma_i^{-1}) \sigma_i S_\alpha =
S_\gamma \omega_i S_\alpha$, and so $\sigma_1,\ldots,\sigma_N$ is a complete system of $(S_\gamma,S_\alpha)$-double coset representatives in
$S_n$ without redundancy.
\end{proof}

\section{Second Specht branching rule for wreath products}
For $n>0$, we can embed $S_m \wr S_{n-1}$ into $S_m \wr S_n$ by mapping $(\sigma;\alpha_1,\ldots,\alpha_{n-1})$, where
$\sigma \in S_{n-1}$, $\alpha_i \in S_m$, to $ (\sigma;\alpha_1,\ldots,\alpha_{n-1},e)$, making use of the canonical embedding of
$S_{n-1}$ into $S_n$.
Hence for $\ul{\lambda} = (\lambda^1,\ldots,\lambda^r)$ an $r$-multipartition of $n$, we can consider the $k(S_m \wr S_{n-1})$-module
\[ S^{\ul{\lambda}}\bigr\downarrow^{m \wr n}_{m \wr (n-1)} \cong
T^{\ul{\lambda}}\bigr\uparrow^{m{\wr}n}_{m{\wr}|\ul{\lambda}|}\bigr\downarrow^{m \wr n}_{m \wr (n-1)}\]
obtained by restricting $S^{\ul{\lambda}}$ from $k(S_m \wr S_n)$ to $k(S_m \wr S_{n-1})$. By Mackey's Theorem we have
\begin{equation}\label{res_Sp_mwrnsub_mackey:eq}
T^{\ul{\lambda}}\bigr\uparrow^{m{\wr}n}_{m{\wr}|\ul{\lambda}|}\bigr\downarrow^{m \wr n}_{m \wr (n-1)}
\cong
\bigoplus_{u \in \mathcal{U}}\bigl(T^{\ul{\lambda}}\bigr)^u\bigl\downarrow^{(m{\wr}|\ul{\lambda}|)^u}_{(m{\wr}|\ul{\lambda}|)^u \,\cap\, m \wr (n-1)}
\bigl\uparrow^{m \wr (n-1)}_{(m{\wr}|\ul{\lambda}|)^u \,\cap\, m \wr (n-1)}
\end{equation}
with minor notational abuses as in the argument for the first branching rule, and where $\mathcal{U}$ represents a complete non-redundant
system of $(S_m \wr S_{|\ul{\lambda}|},S_m \wr S_{n-1})$-double coset representatives in $S_m \wr S_n$. We thus want to find such a set of
double coset representatives. For  $\sigma \in S_n$, let us write $\hat\sigma$ for the element $(\sigma;e,\ldots,e)$ of $S_m \wr S_n$.
Let $\sigma_1,\ldots,\sigma_N$ be a complete non-redundant system of $(S_{|\ul{\lambda}|},S_{n-1})$-double coset representatives in $S_n$. We
claim that $\hat\sigma_1,\ldots,\hat\sigma_N$ is then a complete non-redundant system of $(S_m \wr S_{|\ul{\lambda}|},S_m \wr S_{n-1})$-double
coset representatives in $S_m \wr S_n$. Indeed, if $(\theta;\alpha_1,\ldots,\alpha_n) \in S_m \wr S_n$, then we have
$\theta = \epsilon \sigma_i \delta$
for some $i \in \{1,\ldots,N\}$, $\epsilon \in S_{|\ul{\lambda}|}$ and $\delta \in S_{n-1}$, and it follows that
\[
(\theta;\alpha_1,\ldots,\alpha_n) =
\underbrace{(\epsilon;\alpha_{(1)\sigma_i},\ldots,\alpha_{(n)\sigma_i})}_{\in S_m \wr S_{|\ul{\lambda}|}}
\underbrace{(\sigma_i;e,\ldots,e)}_{=\,\hat\sigma_i}
\underbrace{(\delta;e,\ldots,e)}_{\in S_m \wr S_{n-1}}
\]
which establishes completeness. For non-redundancy, suppose that we have some $i,j$ such that
\[ \bigl(S_m \wr S_{|\ul{\lambda}|}\bigr)\, \hat\sigma_i \, \bigl(S_m \wr S_{n-1}\bigr) =
\bigl(S_m \wr S_{|\ul{\lambda}|}\bigr)\, \hat\sigma_j \, \bigl(S_m \wr S_{n-1}\bigr). \]
Hence $\hat\sigma_i \in \bigl(S_m \wr S_{|\ul{\lambda}|}\bigr)\, \hat\sigma_j \, \bigl(S_m \wr S_{n-1}\bigr)$, so that
we have $\epsilon \in S_{|\ul{\lambda}|}$, $\delta \in S_{n-1}$ and elements $\alpha_i,\beta_i$ of $S_m$ such that
\[
(\sigma_i;e,\ldots,e) =
(\epsilon;\alpha_1,\ldots,\alpha_n)
(\sigma_j;e,\ldots,e)
(\delta;\beta_1,\ldots,\beta_{n-1},e)
\]
from which it follows that $\sigma_i = \epsilon\sigma_j\delta$ and hence that $i=j$. Thus we now seek such $\sigma_1,\ldots,\sigma_N$,
and to do this we shall make use of our work on tableaux.

Now recall that if $\alpha,\gamma$ are compositions of $n$, then we have defined the tableau $\tau^\alpha_\gamma$ to be the tableau of shape
$\alpha$ whose entries, read from left to right across each row in turn starting with the top row, consist of $\gamma_1$ 1's, then
$\gamma_2$ 2's, then $\gamma_3$ 3's, and so on. So for example if $n=9$, $\alpha = (8,1)$ and $\gamma = (3,1,0,2,3)$, then
\[\tau^\alpha_\gamma = \mytab{
{{1,1,1,2,4,4,5,5},{5}}
}.\]
Further, we know by Corollary \ref{tabs_give_cosets:cor} that if we have $\sigma_1,\ldots,\sigma_N \in S_n$ such that
$\tau^\alpha_\gamma\sigma_1,\ldots,\tau^\alpha_\gamma\sigma_N$ is a complete list, with no repetition, of the tableaux of shape $\alpha$ and
type $\gamma$ with weakly increasing rows, then $\sigma_1,\ldots,\sigma_N$ is in fact a complete system of $(S_\gamma,S_\alpha)$-double
coset representatives without redundancy. We now apply this in the case where $\alpha = (n-1,1)$ and $\gamma = |\ul{\lambda}|$ to obtain
our desired system of $(S_{|\ul{\lambda}|},S_{n-1})$-double coset representatives in $S_n$, noting that the subgroup $S_{n-1}$ of $S_n$ is
exactly the Young subgroup $S_{(n-1,1)}$. The following example should serve to illustrate the general argument which we shall give below.

Keep $n=9$, and suppose that $|\ul{\lambda}| = (3,1,0,2,3)$ as above. Then the possible tableaux of shape $(n-1,1)$ and type
$|\ul{\lambda}|$ with weakly increasing rows are
\[
\mytab{
{{1,1,1,2,4,4,5,5},{5}}
}\]\[
\mytab{
{{1,1,1,2,4,5,5,5},{4}}
}\]\[
\mytab{
{{1,1,1,4,4,5,5,5},{2}}
}\]\[
\mytab{
{{1,1,2,4,4,5,5,5},{1}}
}
\]
Thus, a complete non-redundant system of $(S_{|\ul{\lambda}|},S_{(n-1,1)})$-double coset representatives is
$e,(6,9,8,7),(4,9,8,7,6,5),(3,9,8,7,6,5,4)$, recalling that in our action of $S_n$ on tableaux, $\sigma \in S_n$ acts by moving the
contents of the $i$\tss{th} box to the $(i)\sigma$\tss{th} box, where the boxes of a tableau are numbered with the numbers $1,\ldots,n$
from left to right across each row, working from the top row to the bottom row.

The general case works in exactly the same way as the example. Indeed, recall that $\ul{\lambda} = (\lambda^1,\ldots,\lambda^r)$. For
$i=1,\ldots,r$ we let $b_i = |\lambda^1|+\cdots+|\lambda^i|$, so that we have a sequence $0 \leq b_1 \leq b_2 \leq \cdots \leq b_r = n$.
Then for each $i=1,\ldots,r$ such that $b_i \neq 0$ we define an element $\rho_i$ of $S_n$ by letting
\[
\rho_i = \begin{cases}
(b_i,n,n-1,\ldots,b_i+1) \quad\text{if $b_i < n$}\\
e \quad\text{if $b_i = n$}
\end{cases}
\]
(where $e$ is the identity element). By letting $i$ run through all $1,\ldots,r$ such that $|\lambda^i| > 0$, we obtain a complete list
of all the distinct $\rho_i$ without repetition. As in the above example, we see that the set of all tableaux
$\tau^{(n-1,1)}_{|\ul{\lambda}|}\rho_i$ for $i$ such that $|\lambda^i| > 0$ forms a complete list of all of the tableaux of shape $(n-1,1)$
and type $|\ul{\lambda}|$ with weakly increasing rows. Hence by Corollary \ref{tabs_give_cosets:cor} we see that the collection of all
$\rho_i$ for $i$ such that $|\lambda^i| > 0$ forms a complete non-redundant system of $(S_{|\ul{\lambda}|},S_{n-1})$-double coset
representatives in $S_n$, and hence the collection of all $\hat\rho_i$ for $i$ such that $|\lambda^i| > 0$ forms a complete non-redundant
system of $(S_m \wr S_{|\ul{\lambda}|},S_m \wr S_{n-1})$-double coset representatives in $S_m \wr S_n$.

Looking back to \eqref{res_Sp_mwrnsub_mackey:eq}, we see that we want to understand the module
\[
\bigl(T^{\ul{\lambda}}\bigr)^{\hat\rho_i}\bigl\downarrow^{(m{\wr}|\ul{\lambda}|)^{\hat\rho_i}}_{(m{\wr}|\ul{\lambda}|)^{\hat\rho_i} \,\cap\, m \wr (n-1)}
\bigl\uparrow^{m \wr (n-1)}_{(m{\wr}|\ul{\lambda}|)^{\hat\rho_i} \,\cap\, m \wr (n-1)}
\]
for $i$ such that $|\lambda^i|>0$. Our first step in doing so will be to understand the subgroup
$\bigl(S_m \wr S_{|\ul{\lambda}|}\bigr)^{\hat\rho_i}\cap\bigl(S_m \wr S_{n-1}\bigr)$ of $S_m \wr S_n$ and its action on the module
$\bigl(T^{\ul{\lambda}}\bigr)^{\hat\rho_i}$.

So choose $i$ such that $|\lambda^i|>0$.
It is easy to show directly that $\bigl(S_m \wr S_{|\ul{\lambda}|}\bigr)^{\hat\rho_i}$ is equal to
$S_m \wr \bigl(S_{|\ul{\lambda}|}\bigr)^{\rho_i}$. Thus we have
\[\bigl(S_m \wr S_{|\ul{\lambda}|}\bigr)^{\hat\rho_i}\cap\bigl(S_m \wr S_{n-1}\bigr)
=
S_m \wr \bigl(S_{|\ul{\lambda}|}\bigr)^{\rho_i}\cap\bigl(S_m \wr S_{n-1}\bigr)\]
and it is easy to show directly that $S_m \wr \bigl(S_{|\ul{\lambda}|}\bigr)^{\rho_i}\cap\bigl(S_m \wr S_{n-1}\bigr)$
is equal to the subgroup of $S_m \wr S_n$ consisting of all elements of the form
\begin{equation}\label{elt_of_conj_intersect:eq}
(\sigma; \alpha_1,\ldots,\alpha_{n-1},e)
\end{equation}
where $\sigma$ is an element of the subgroup
$\bigl(S_{|\ul{\lambda}|}\bigr)^{\rho_i}\cap S_{n-1}$ of $S_n$
and $\alpha_i \in S_m$. We thus wish to understand the subgroup $\bigl(S_{|\ul{\lambda}|}\bigr)^{\rho_i}\cap S_{n-1}$ of $S_n$.
By Proposition \ref{tab_stab:prop}, $\bigl(S_{|\ul{\lambda}|}\bigr)^{\rho_i}$ is the stabilizer (under
the action of $S_n$) of the tableau $\tau^{(n-1,1)}_{|\ul{\lambda}|}\rho_i$. It is easy to see that the tableau
$\tau^{(n-1,1)}_{|\ul{\lambda}|}\rho_i$ is the unique tableau of shape $(n-1,1)$ and type $|\ul{\lambda}|$ with weakly increasing rows which
has an $i$ in the box on the second row; such tableaux are illustrated in the above example. For any subset $\Omega$ of $\{1,\ldots,n\}$,
let us write $S(\Omega)$ to denote the subgroup of
$S_n$ consisting of all permutations which fix any number not lying in $\Omega$. We easily see that the stabilizer of the tableau
$\tau^{(n-1,1)}_{|\ul{\lambda}|}\rho_i$ is the subgroup $X^i_{|\ul{\lambda}|}$ of $S_n$, where we define (recalling that $|\lambda^i|>0$ and hence
$b_i > b_{i-1}$, where $b_0$ is taken to be 0)
\begin{multline*}
X^i_{|\ul{\lambda}|} = S\bigl(\{1,\ldots,b_1\}\bigr) \times S\bigl(\{b_1+1,\ldots,b_2\}\bigr) \times \cdots\\
\times S\bigl(\{b_{i-1}+1,\ldots,b_i-1,n\}\bigr)\times
S\bigl(\{b_i,\ldots,b_{i+1}-1\}\bigr)\times S\bigl(\{b_{i+1},\ldots,b_{i+2}-1\}\bigr)\times\\
\cdots \times S\bigl(\{b_{r-1},\ldots,b_r-1 \,=\, n-1\}\bigr)
\end{multline*}
(note that here we are using the $\times$ symbol to denote an \emph{internal} direct product of subgroups, and that if $b_i = b_{i+1}$ then
$\{b_i,\ldots,b_{i+1}-1\}$ represents the empty set, and that if $b_i = b_{i-1}+1$ then $\{b_{i-1}+1,\ldots,b_i-1,n\} = \{n\}$), and hence
$\bigl(S_{|\ul{\lambda}|}\bigr)^{\rho_i} = X^i_{|\ul{\lambda}|}$.
We now introduce a small piece of notation.
Indeed, if $\gamma = (\gamma_1,\ldots,\gamma_r)$ is a composition of $n$, and $i \in \{1,\ldots,r\}$ such that $\gamma_i > 0$, then we
write $[\gamma]_i$ for the composition $(\gamma_1,\ldots,\gamma_{i-1},\gamma_i -1, \gamma_{i+1},\gamma_r)$ of $n-1$.
We see that $X^i_{|\ul{\lambda}|}\cap S_{n-1}$ is the subgroup
\begin{multline*}
S\bigl(\{1,\ldots,b_1\}\bigr) \times S\bigl(\{b_1+1,\ldots,b_2\}\bigr) \times \cdots\\
\times S\bigl(\{b_{i-1}+1,\ldots,b_i-1\}\bigr)\times S\bigl(\{b_i,\ldots,b_{i+1}-1\}\bigr)\times S\bigl(\{b_{i+1},\ldots,b_{i+2}-1\}\bigr)
\times\\
\cdots \times S\bigl(\{b_{r-1},\ldots,b_r-1 \,=\, n-1\}\bigr)
\end{multline*}
of $S_n$, and under our embedding of $S_{n-1}$ into $S_n$
this is exactly the subgroup $S_{[|\ul{\lambda}|]_i}$ of $S_{n-1}$.
Hence, recalling that we are viewing $S_m \wr S_{n-1}$ as a subgroup
of $S_m \wr S_n$ via the embedding $(\sigma;\alpha_1,\ldots,\alpha_{n-1}) \longmapsto (\sigma;\alpha_1,\ldots,\alpha_{n-1},e)$, we see that
the subgroup $\bigl(S_m \wr S_{|\ul{\lambda}|}\bigr)^{\hat\rho_i}\cap\bigl(S_m \wr S_{n-1}\bigr)$ of $S_m \wr S_n$ is equal to the subgroup
$S_m \wr S_{[|\ul{\lambda}|]_i}$ of the subgroup $S_m \wr S_{n-1}$ of $S_m \wr S_n$.

We now turn our attention to the action of
$\bigl(S_m \wr S_{|\ul{\lambda}|}\bigr)^{\hat\rho_i}\cap\bigl(S_m \wr S_{n-1}\bigr)$
on the
$k\bigl(S_m \wr S_{|\ul{\lambda}|}\bigr)^{\hat\rho_i}$-module $\left(T^{\ul{\lambda}}\right)^{\hat\rho_i}$.
We know by the definition of conjugate modules that $\left(T^{\ul{\lambda}}\right)^{\hat\rho_i}$ is the module formed by equipping
$T^{\ul{\lambda}}$ with the $k\bigl(S_m \wr S_{|\ul{\lambda}|}\bigr)^{\hat\rho_i}$-action $\ast$ given for $x \in T^{\ul{\lambda}}$ and
$y \in \bigl(S_m \wr S_{|\ul{\lambda}|}\bigr)^{\hat\rho_i}$ by $x \ast y = x (\hat\rho_i\,y\,\hat\rho_i^{-1})$ (where the action on the
right-hand side is the action of $S_m \wr S_{|\ul{\lambda}|}$ on $T^{\ul{\lambda}}$, noting that $\hat\rho_i\,y\,\hat\rho_i^{-1}$ does indeed
lie in $S_m \wr S_{|\ul{\lambda}|}$). Thus to calculate the action of an element
\[(\sigma;\alpha_1,\ldots,\alpha_{n-1},e) \in \bigl(S_m \wr S_{|\ul{\lambda}|}\bigr)^{\hat\rho_i}\cap\bigl(S_m \wr S_{n-1}\bigr)\]
on the module $\left(T^{\ul{\lambda}}\right)^{\hat\rho_i}$, we need to calculate
$\hat\rho_i(\sigma;\alpha_1,\ldots,\alpha_{n-1},e)\hat\rho_i^{-1}$. We have
\begin{align*}
\hat\rho_i(\sigma;\alpha_1,\ldots,\alpha_{n-1},e)\hat\rho_i^{-1}
&= (\rho_i;e,\ldots,e)(\sigma;\alpha_1,\ldots,\alpha_{n-1},e)(\rho_i^{-1};e,\ldots,e)\\
&= (\rho_i\sigma\rho_i^{-1};\alpha_{(1)\rho_i},\ldots,\alpha_{(n)\rho_i})\qquad \text{(taking $\alpha_n = e$)}\\
&= (\rho_i\sigma\rho_i^{-1};\alpha_1,\alpha_2,\ldots,\alpha_{b_i-1},e,\alpha_{b_i},\alpha_{b_i+1},\\
&\hspace{13em}\ldots,\alpha_{n-2},\alpha_{n-1}).
\end{align*}
But by our description \eqref{elt_of_conj_intersect:eq} of the elements of
$\bigl(S_m \wr S_{|\ul{\lambda}|}\bigr)^{\hat\rho_i}\cap\bigl(S_m \wr S_{n-1}\bigr)$, we see that
$\sigma \in \bigl(S_{|\ul{\lambda}|}\bigr)^{\rho_i}\cap S_{n-1}$, which implies that
$\rho_i\sigma\rho_i^{-1} \in S_{|\ul{\lambda}|}\cap \bigl(S_{n-1}\bigr)^{\rho^{-1}_i}$. By direct calculation, any element of
$\bigl(S_{n-1}\bigr)^{\rho^{-1}_i}$ fixes $b_i$, and hence we see that $\rho_i\sigma\rho_i^{-1}$ is an element of
$S_{|\ul{\lambda}|}$ which fixes $b_i$.
Now we know that the subgroup $S_{|\ul{\lambda}|}$ of $S_n$ has an internal direct product factorisation
\begin{multline*}
S\bigl(\{1,\ldots,b_1\}\bigr) \times S\bigl(\{b_1+1,\ldots,b_2\}\bigr) \times \cdots\\
\cdots\times S\bigl(\{b_{i-1}+1,\ldots,b_i\}\bigr)\times S\bigl(\{b_i+1,\ldots,b_{i+1}\}\bigr)\times\cdots\\
\cdots \times S\bigl(\{b_{r-1}+1,\ldots,b_r \,=\, n\}\bigr).
\end{multline*}
Thus any element $\pi$ of $S_{|\ul{\lambda}|}$ has a unique factorisation $\pi = \theta_1\cdots\theta_r$ where
$\theta_j \in S\bigl(\{b_{j-1}+1,\ldots,b_j\}\bigr)$ (with $b_0$ taken to be 0). We thus see that $\rho_i\sigma\rho_i^{-1}$ has such a
factorisation $\rho_i\sigma\rho_i^{-1} = \theta_1\cdots\theta_r$, where $\theta_i$ fixes $b_i$. 
Thus we see that our element
$(\sigma;\alpha_1,\ldots,\alpha_{n-1},e)$ of $\bigl(S_m \wr S_{|\ul{\lambda}|}\bigr)^{\hat\rho_i}\cap\bigl(S_m \wr S_{n-1}\bigr)$ acts on the
module $\bigl(T^{\ul{\lambda}}\bigr)^{\hat\rho_i}$ as the element
\[
\bigl(\theta_1\cdots\theta_r;\,\alpha_1,\alpha_2,\ldots,\alpha_{b_i-1},e,\alpha_{b_i},\alpha_{b_i+1},\ldots,\alpha_{n-2},\alpha_{n-1}\bigr)
\]
of $S_m \wr S_{|\ul{\lambda}|}$ acts on $T^{\ul{\lambda}}$ (recalling that $\bigl(T^{\ul{\lambda}}\bigr)^{\hat\rho_i}$ and $T^{\ul{\lambda}}$ are equal
as $k$-vector spaces). But we know that $\bigl(S_m \wr S_{|\ul{\lambda}|}\bigr)^{\hat\rho_i}\cap\bigl(S_m \wr S_{n-1}\bigr)$ is equal to the
subgroup $S_m \wr S_{[|\ul{\lambda}|]_i}$ of the subgroup $S_m \wr S_{n-1}$ of $S_m \wr S_n$, and we now see that if we identify
$S_m \wr S_{[|\ul{\lambda}|]_i}$ with
\begin{multline*}
(S_m \wr S_{|\lambda^1|})\times(S_m \wr S_{|\lambda^2|})\times\cdots\\
\cdots\times
(S_m \wr S_{|\lambda^{i-1}|})\times(S_m \wr S_{|\lambda^i|-1})\times(S_m \wr S_{|\lambda^{i+1}|})\times\cdots\times
(S_m \wr S_{|\lambda^r|})
\end{multline*}
in the canonical way, then by the definition of the $k(S_m \wr S_{|\ul{\lambda}|})$-module $T^{\ul{\lambda}}$, the
$k(S_m \wr S_{[|\ul{\lambda}|]_i})$-module
\[
\bigl(T^{\ul{\lambda}}\bigr)^{\hat\rho_i}\bigl\downarrow^{(m{\wr}|\ul{\lambda}|)^{\hat\rho_i}}_{(m{\wr}|\ul{\lambda}|)^{\hat\rho_i} \,\cap\, m \wr (n-1)}
=
\bigl(T^{\ul{\lambda}}\bigr)^{\hat\rho_i}\bigl\downarrow^{(m{\wr}|\ul{\lambda}|)^{\hat\rho_i}}_{m \wr [|\ul{\lambda}|]_i}
\]
is isomorphic to
\begin{multline}\label{T_ul_lambd_isom_in_place_res:eq}
\left(\bigl(S^{\mu^1}\bigr)^{\widetilde\boxtimes |\lambda^1|}\oslash S^{\lambda^1}\right)\boxtimes\cdots\boxtimes
\left(\bigl(S^{\mu^i}\bigr)^{\widetilde\boxtimes |\lambda^i|}\oslash S^{\lambda^i}\right)\Bigl\downarrow^{m \wr |\lambda^i|}_{m \wr (|\lambda^i|-1)}
\boxtimes\cdots\\
\cdots\boxtimes
\left(\bigl(S^{\mu^r}\bigr)^{\widetilde\boxtimes |\lambda^r|}\oslash S^{\lambda^r}\right).
\end{multline}
Thus, we want to investigate the $k(S_m \wr S_{|\lambda^i|-1})$-module
\[\left(\bigl(S^{\mu^i}\bigr)^{\widetilde\boxtimes |\lambda^i|}\oslash S^{\lambda^i}\right)\Bigl\downarrow^{m \wr |\lambda^i|}_{m \wr (|\lambda^i|-1)}.\]
Now the restriction operation $\Bigl\downarrow^{m \wr |\lambda^i|}_{m \wr (|\lambda^i|-1)}$ may be expressed as
\[\Bigl\downarrow^{m \wr |\lambda^i|}_{m \wr (|\lambda^i|-1,1)}\Bigl\downarrow^{m \wr (|\lambda^i|-1,1)}_{m \wr (|\lambda^i|-1)},\]
where, we recall, $m \wr (|\lambda^i|-1,1)$ represents the subgroup $S_m \wr S_{(|\lambda^i|-1,1)}$ of $S_m \wr S_{|\lambda^i|}$ consisting of all
elements of the form $(\sigma;\alpha_1,\ldots,\alpha_n)$ for $\alpha_i \in S_m$ and $\sigma \in S_{(|\lambda^i|-1,1)}$, while
$m \wr (|\lambda^i|-1)$ represents the subgroup $S_m \wr S_{(|\lambda^i|-1)}$ of $S_m \wr S_{|\lambda^i|}$ consisting of all elements of the form
$(\sigma;\alpha_1,\ldots,\alpha_{n-1},e)$ for $\alpha_i \in S_m$ and $\sigma \in S_{(|\lambda^i|-1,1)}$. Now we have by
Proposition \ref{wreath_ind_res:prop} that
\[
\left(\bigl(S^{\mu^i}\bigr)^{\widetilde\boxtimes |\lambda^i|}\oslash S^{\lambda^i}\right)\Bigl\downarrow^{m \wr |\lambda^i|}_{m \wr (|\lambda^i|-1,1)}
=
\bigl(S^{\mu^i}\bigr)^{\widetilde\boxtimes |\lambda^i|}\Bigl\downarrow^{m \wr |\lambda^i|}_{m \wr (|\lambda^i|-1,1)}\:\oslash\:
S^{\lambda^i}\Bigl\downarrow^{|\lambda^i|}_{(|\lambda^i|-1,1)}.
\]
Upon further restriction to $S_m \wr S_{(|\lambda^i|-1)}$, we see that this is isomorphic to the direct sum of $\mathrm{dim}_k(S^{\mu^i})$
copies of
\[
\bigl(S^{\mu^i}\bigr)^{\widetilde\boxtimes |\lambda^i|-1}\,\oslash\,S^{\lambda^i}\!\Bigl\downarrow^{|\lambda^i|}_{|\lambda^i|-1}.
\]
It now follows by Theorem \ref{jam_branch_rule_sn:thm} and the fact that $- \oslash -$ preserves filtrations (see above)
that, if for any partition $\epsilon$ we define
$\mathrm{R}(\epsilon)$ to be the set of all partitions of $|\epsilon|-1$ which may be obtained from $\epsilon$ by removing a box, then
we have a filtration of
$\left(\bigl(S^{\mu^i}\bigr)^{\widetilde\boxtimes |\lambda^i|}\oslash S^{\lambda^i}\right)\Bigl\downarrow^{m \wr |\lambda^i|}_{m \wr (|\lambda^i|-1)}$
by modules $\bigl(S^{\mu^i}\bigr)^{\widetilde\boxtimes |\lambda^i|-1}\oslash S^{\delta}$ for $\delta \in \mathrm{R}(\lambda^i)$ where
$\bigl(S^{\mu^i}\bigr)^{\widetilde\boxtimes |\lambda^i|-1}\oslash S^{\delta}$ has multiplicity $\mathrm{dim}_k(S^{\mu^i})$.
Using \eqref{T_ul_lambd_isom_in_place_res:eq}, it now follows that we have a filtration of the
$k(S_m \wr S_{[|\ul{\lambda}|]_i})$-module
\[\bigl(T^{\ul{\lambda}}\bigr)^{\hat\rho_i}\bigl\downarrow^{(m{\wr}|\ul{\lambda}|)^{\hat\rho_i}}_{(m{\wr}|\ul{\lambda}|)^{\hat\rho_i} \,\cap\, m \wr (n-1)}
=
\bigl(T^{\ul{\lambda}}\bigr)^{\hat\rho_i}\bigl\downarrow^{(m{\wr}|\ul{\lambda}|)^{\hat\rho_i}}_{m \wr [|\ul{\lambda}|]_i}\]
by modules
$T^{\ul{\delta}}$ for $\ul{\delta}$ an $r$-multipartition of $n-1$ such that $\delta^j = \lambda^j$ for $j \neq i$ and
$\delta^i \in \mathrm{R}(\lambda^i)$, where $T^{\ul{\delta}}$ has multiplicity $\mathrm{dim}_k(S^{\mu^i})$.
By exactness of the functor $\bigl\uparrow^{m \wr (n-1)}_{m \wr [|\ul{\lambda}|]_i}$, it now follows that we have a filtration of the
$k(S_m \wr S_{n-1})$-module
\[
\bigl(T^{\ul{\lambda}}\bigr)^{\hat\rho_i}\bigl\downarrow^{(m{\wr}|\ul{\lambda}|)^{\hat\rho_i}}_{(m{\wr}|\ul{\lambda}|)^{\hat\rho_i} \,\cap\, m \wr (n-1)}
\bigl\uparrow^{m \wr (n-1)}_{(m{\wr}|\ul{\lambda}|)^{\hat\rho_i} \,\cap\, m \wr (n-1)}
=
\bigl(T^{\ul{\lambda}}\bigr)^{\hat\rho_i}\bigl\downarrow^{(m{\wr}|\ul{\lambda}|)^{\hat\rho_i}}_{m \wr [|\ul{\lambda}|]_i}
\bigl\uparrow^{m \wr (n-1)}_{m \wr [|\ul{\lambda}|]_i}
\]
by modules $S^{\ul{\delta}}$ for $\ul{\delta}$ an $r$-multipartition of $n-1$ such that $\delta^j = \lambda^j$ for $j \neq i$ and
$\delta^i \in \mathrm{R}(\lambda^i)$, where $S^{\ul{\delta}}$ has multiplicity $\mathrm{dim}_k(S^{\mu^i})$.
Referring back to the decomposition \eqref{res_Sp_mwrnsub_mackey:eq}, we now see that we have proved the following result, which is our
desired Specht branching rule.

\begin{theorem}
Let $n>0$, and let $\ul{\lambda}$ be an $r$-multipartition of $n$. Then we have a filtration of the $k(S_m \wr S_{n-1})$-module
$\left.S^{\ul{\lambda}}\right\downarrow^{m \wr n}_{m \wr (n-1)}$ by Specht modules $S^{\ul{\delta}}$ for $r$-multipartitions $\ul{\delta}$ of $n-1$.
For a multipartition $\ul{\delta}$ of $n-1$, if $\ul{\delta}$ may be obtained from $\ul{\lambda}$ by removing a single box from the
partition $\lambda^i$ for some $i$ (while leaving all other partitions $\lambda^j$ unchanged), then $S^{\ul{\delta}}$ has with
multiplicity $\mathrm{dim}_k(S^{\mu^i})$ in the filtration, and otherwise $S^{\ul{\delta}}$ has multiplicity zero in the filtration.
\end{theorem}

We note that the multiplicities $\mathrm{dim}_k(S^{\mu^i})$ occurring in this filtration have a simple and elegant combinatorial
interpretation via the \emph{hook length formula} (see for example \cite[chapter 20]{JAMSN}), from which we see that they are in fact
independent of the field $k$. We also note the similarity of this result to Theorem \ref{jam_branch_rule_sn:thm}.

\end{document}